\documentclass[11pt]{article}
\usepackage[utf8]{inputenc} 
\usepackage[english]{babel}  
\usepackage{amssymb,amsmath,amsthm,amsfonts} 
\numberwithin{equation}{section}
\usepackage{mathtools}
\mathtoolsset{showonlyrefs}
\usepackage{color}

\newtheorem{theorem}{Theorem}[section]
\newtheorem{lemma}[theorem]{Lemma}

\newtheorem{corollary}[theorem]{Corollary}
\newtheorem{remark}[theorem]{Remark}

\let\originalleft\left
\let\originalright\right
\renewcommand{\left}{\mathopen{}\mathclose\bgroup\originalleft}
\renewcommand{\right}{\aftergroup\egroup\originalright}
\newlength{\leftstackrelawd}
\newlength{\leftstackrelbwd}
\def\leftstackrel#1#2{\settowidth{\leftstackrelawd}%
	{${{}^{#1}}$}\settowidth{\leftstackrelbwd}{$#2$}%
	\addtolength{\leftstackrelawd}{-\leftstackrelbwd}%
	\leavevmode\ifthenelse{\lengthtest{\leftstackrelawd>0pt}}%
	{\kern-.5\leftstackrelawd}{}\mathrel{\mathop{#2}\limits^{#1}}}

\newcommand{\brc}[1]{  \left\{#1\right\} } 
  
\newcommand{\norm}[1]{  \left\|#1\right\| } 
\newcommand{\pare}[1]{\left(#1\right)}    
\newcommand{\cor}[1]{  \left[#1\right] }      
\newcommand{\abs}[1]{  \left\vert#1\right\vert }     
     
\newcommand{\bb}[1]{\mathbb{#1}}

\def\mathcolor#1#{\@mathcolor{#1}}
\def\@mathcolor#1#2#3{%
	\protect\leavevmode
	\begingroup
	\color#1{#2}#3%
	\endgroup
}

\def\11{\textbbm{1}}

\newcommand*{\avint}{\mathop{\ooalign{$\int$\cr$-$}}}

\def\vint_#1{\mathchoice%
          {\mathop{\kern 0.2em\vrule width 0.6em height 0.69678ex depth -0.58065ex
                  \kern -0.8em \intop}\nolimits_{\kern -0.4em#1}}%
          {\mathop{\kern 0.1em\vrule width 0.5em height 0.69678ex depth -0.60387ex
                  \kern -0.6em \intop}\nolimits_{#1}}%
          {\mathop{\kern 0.1em\vrule width 0.5em height 0.69678ex
              depth -0.60387ex
                  \kern -0.6em \intop}\nolimits_{#1}}%
          {\mathop{\kern 0.1em\vrule width 0.5em height 0.69678ex depth -0.60387ex
                  \kern -0.6em \intop}\nolimits_{#1}}}
\def\vintslides_#1{\mathchoice%
          {\mathop{\kern 0.1em\vrule width 0.5em height 0.697ex depth -0.581ex
                  \kern -0.6em \intop}\nolimits_{\kern -0.4em#1}}%
          {\mathop{\kern 0.1em\vrule width 0.3em height 0.697ex depth -0.604ex
                  \kern -0.4em \intop}\nolimits_{#1}}%
          {\mathop{\kern 0.1em\vrule width 0.3em height 0.697ex depth -0.604ex
                  \kern -0.4em \intop}\nolimits_{#1}}%
          {\mathop{\kern 0.1em\vrule width 0.3em height 0.697ex depth -0.604ex
                  \kern -0.4em \intop}\nolimits_{#1}}}

\newcommand{\kint}{\vint}

\newcommand{\aveint}[2]{\mathchoice%
          {\mathop{\kern 0.2em\vrule width 0.6em height 0.69678ex depth -0.58065ex
                  \kern -0.8em \intop}\nolimits_{\kern -0.45em#1}^{#2}}%
          {\mathop{\kern 0.1em\vrule width 0.5em height 0.69678ex depth -0.60387ex
                  \kern -0.6em \intop}\nolimits_{#1}^{#2}}%
          {\mathop{\kern 0.1em\vrule width 0.5em height 0.69678ex depth -0.60387ex
                  \kern -0.6em \intop}\nolimits_{#1}^{#2}}%
          {\mathop{\kern 0.1em\vrule width 0.5em height 0.69678ex depth -0.60387ex
                  \kern -0.6em \intop}\nolimits_{#1}^{#2}}}

\begin{document}

\title{\textbf{Mixing local and nonlocal evolution equations}}
\author{Monia Capanna and Julio D. Rossi}

%\date{}

\maketitle

\begin{abstract}
In this paper we study the homogenization of a stochastic process and its associated
evolution equations in which we mix a local
part (given by a Brownian motion with a reflection on the boundary) and 
a nonlocal part (given by a jump process with a smooth kernel). 
We consider a sequence of partitions of the (fixed) spacial domain into two parts (local and nonlocal) that are mixed in such a way that they both have positive
density at every point in the limit. Under adequate hypotheses on the sequence of partitions,
we prove convergence of the associated densities (that are solutions
to an evolution equation with coupled local and nonlocal parts in two different regions of the domain) 
to the unique solution to a limit evolution system 
in which the local part disappears and the nonlocal part survives but divided into
two different components.  We also obtain convergence in distributions of the 
processes associated to the partitions and prove that the limit process has a density pair that
coincides with the limit of the densities.
\end{abstract}

%\date{ }

\noindent{\makebox[1in]\hrulefill}\newline2010 \textit{Mathematics Subject
Classification.} 	35B27, 45K05, 60J60, 35K15.
% 	60J60  	Diffusion processes
% 45K05  	Integro-partial differential equations
% 35K15  	Initial value problems for second-order parabolic equations
%35B27  	Homogenization; equations in media with periodic structure
% 35R11, %Fractional partial differential equations
%35S10, %Initial value problems for PsDO
%35B65, %Smoothness and regularity of solutions
%35K55, %Nonlinear PDE of parabolic type
%35B40, Asymptotic behavior of solutions
%26A33, %Fractional derivatives and integrals
%35A05. %General existence and uniqueness theorems
%35K65, %Parabolic partial differential equations of degenerate type
%76S05 %Flows in porous media; filtration; seepage

\renewcommand{\theequation}{\arabic{section}.\arabic{equation}}

\section{Introduction}\label{sec.intro}

If you think about a linear diffusion equation, the first example that will come to your mind is the classical heat equation
\begin{equation} \label{heat.equation}
\frac{\partial u}{\partial t} (x,t) = \Delta u (x,t).
\end{equation}
This local partial differential equation is naturally associated with the Brownian motion $B_t$,
in the sense that \eqref{heat.equation} is the equation that appears for the probability density when one considers a 
particle that moves according to Brownian motion. For \eqref{heat.equation} the initial condition $u(x,0)=u_0(x)$ gives the initial distribution of the particle
and if we consider Neumann boundary conditions in a bounded domain we add the condition that the particle is reflected on the boundary.

If you go one step further and consider nonlocal diffusion problems,
one possible model is
\begin{equation} \label{eq:nonlocal.equation}
\frac{\partial u}{\partial t} (x,t) = \int J(x,y)(u(y,t)-u(x,t))
\, {d}y,
\end{equation}
where the kernel $J(x,y)$ is nonnegative, continuous in both variables and symmetric with
$\int J(x,y)\, dy  =1$ (these hypotheses on $J$ will be assumed in what follows). Notice that the diffusion of the density $u$ at a point $x$ and time $t$ depends on the values of $u$ at all points in the set $\mathop{\rm supp}(J(x,\cdot))$, which is what makes the diffusion operator nonlocal. Evolution equations of this form
and variations of it have been recently widely used to model
diffusion processes;  see for instance \cite{BLGneu,BCh, CF, CERW, delia1} and the book \cite{ElLibro}. As stated in \cite{F}, if $u(x,t)$ is thought of
as the density of a single population at the point $x$ at time $t$, and
$J(x,y)$ is regarded as the probability distribution of jumping from location
$y$ to location~$x$, then  the
rate at which individuals are arriving to position $x$ from all other places is given by $\int J(y,x)u(y,t)\,{d}y$, while  the rate at which they are
leaving location $x$ to travel to all other sites is given by $-\int J(y,x)u(x,t)\, {d}y=-u(x,t)$. Therefore,  in the
absence of external or internal sources, the
density $u$ satisfies equation (\ref{eq:nonlocal.equation}).
In this case there is also a process that governs the
evolution problem, namely, the jump process with probability of jumping from $x$ to $y$ given by $J(x,y)$.
Notice we are not assuming that
the nonlocal part is given by a convolution, that is, $J(x,y) = G(x-y)$.

Now we consider an open bounded domain $\Omega$ such that $\abs{\partial\Omega}=0$, we split its closure
$\overline\Omega$ into two disjoint pieces,
$\overline{\Omega} = A \cup B$, $A\cap B= \emptyset$ and consider an evolution 
that has local and nonlocal features according to the spacial position.
We will refer to $A$ as the nonlocal subdomain and $B$ as the local one.
Informally, let us consider a particle that 
may jump (according to the probability density $J(x,y)$ that generates the nonlocal evolution equation) 
when the initial point or the target point, $x$ or $y$,
belongs to the nonlocal region $A\subset \overline\Omega$ and it also
moves according to Brownian motion 
in the other subdomain $B= \overline\Omega \setminus A$ (with a reflexion at the boundary of $B$, hence we
assume Lipchitz regularity for $\partial B$). The associated evolution equation
has two main parts, one driven by the Laplacian with Neumann boundary conditions
(in the set $B$) and another one driven by a nonlocal operator (in $A$). See \eqref{gen.intro} and \eqref{evol.intro} below.

Here we look for an homogenization procedure of this local/nonlocal setting.
We take a sequence of partitions
$A_n$, $B_n$ of the fixed ambient space $\overline{\Omega}$ such that $\overline{\Omega} = A_n \cup B_n$, $A_n\cap B_n= \emptyset$,
$B_n$ is open, has a Lipchitz boundary (consequently $\abs{\partial B_n}=\abs{\partial A_n}=0$) and
\begin{equation} \label{cond.sets}
\begin{array}{l}
\bullet \  \chi_{B_n} (x) \rightharpoonup \theta (x), \qquad \mbox{ weakly in } L^\infty (\overline\Omega), 
\qquad \mbox{with } 0< \theta(x) <1, \\[10pt]
\bullet \mbox{ The connected components of $B_n$, $\{B_n^j\}_j$, verify }  \\[10pt]
\qquad \qquad  \max_j \left\{\mbox{diam} (B^j_n) \right\} \to 0, \mbox{ as } n \to \infty.
\end{array}
\end{equation}
Note that $\chi_{B_n} (x) \rightharpoonup \theta(x)$ implies
$\chi_{A_n} (x) \rightharpoonup 1-\theta(x)$. Also notice that the convergences $\chi_{B_n} (x) \rightharpoonup \theta (x)$ and $\chi_{A_n} (x) \rightharpoonup 1-\theta(x)$ imply that
$$
\int_E \chi_{B_n} (x)  \, dx\to \int_E \theta (x)  \, dx \qquad \mbox{and} \qquad 
\int_E \chi_{A_n} (x)  \, dx\to \int_E (1-\theta (x))  \, dx, \qquad \mbox{ as } n \to +\infty.
$$
Since we assume that $0< \theta(x) <1$ we have that for every $E\subseteq \overline\Omega$ with positive 
measure, $|E|>0$, it holds that $|E\cap A_n|>0$ and also $|E\cap B_n|>0$ for $n$ large which reflects the fact
that we are mixing the two sets $A_n$, $B_n$, in the whole $\overline{\Omega}$.

We study the evolution of a particle that moves into $\overline{\Omega}$ in a way that we describe in detail as follows:
first, we introduce $\brc{E_k}_{k\in \bb N}$ a family of independent exponential random variables and $J\in C\pare{\overline\Omega \times \overline\Omega}$ a symmetric kernel in $\overline\Omega \times \overline\Omega$ such that $\int_\Omega J(x, y)dy=1$ for all $x\in \overline\Omega$. Fixing $\tau_0=0$, we define recursively the random times
$\tau_k=\tau_{k-1}+E_k$, for every $k\in \bb N$.
From now on $X_n\pare{t}$ denotes the position of the particle at time $t$ and $X_n^-\pare{t}:=\lim_{s\to t^-}X_n\pare{s}$ almost surely. 
The evolution of the particle is described in the following terms: 
between two consecutive times $t\in (\tau_k, \tau_{k+1})$ whenever the process $X_n\pare{t}$ is in $B_n$, the particle moves like a Brownian motion which is reflected at the boundary $\partial B_n$, 
while when it is in $A_n$ it rests still. At the times $\{\tau_k\}$ the particle chooses a site $y\in \overline{\Omega}$ according to the kernel $J\pare{X^-_n(\tau_k), y}$ and it jumps on it with the following restriction: from the sites in $B_n$ only the jumps to the sites in $A_n$ are allowed (if the target point $y$ belongs to $B_n$ when 
$X^-_n(\tau_k) \in B_n$ then the 
jump is suppressed and the particle just continues moving according to Brownian motion from its current position).  
The initial position $X_n(0)$ is described in terms of a given distribution $u_0$ in $\overline\Omega$. More precisely, we suppose that
\begin{align}
P\pare{X_n\pare{0}\in E}=\int_{E} u_0(z) \, dz, \end{align}
for every measurable set $E\subseteq \overline\Omega$.

The process $X_n(t)$ is a Markov process whose generator ${L}_n$ is defined on functions  
$$f\in \mathcal S_n:=\left\{f:\overline\Omega\to\bb R\; : f\in C\pare{A_n}\cap C^2\pare{B_n}\, : \frac{\partial f}{\partial \eta}\mid_{\partial B_n}=0\right\}$$ 
(here and in what follows $\eta$ stands for the exterior unit normal vector to $\partial B_n$) and is given by
\begin{align}\label{gen.intro}
{L}_n f(x)=
&\chi_{A_n}\pare{x}\int_\Omega J\pare{x,y}\pare{f\pare{y}-f\pare{x}}dy\\[10pt]
&+\chi_{B_n}\pare{x}\int_\Omega \chi_{A_n}\pare{y}J\pare{x,y}\pare{f\pare{y}-f\pare{x}}dy+\frac{1}{2}\chi_{B_n}\pare{x} \Delta f \pare{x}.
\end{align}

The associated evolution problem (whose solution is the density of the process $X_n$, see 
Corollary \ref{corol.densidad}) reads as
\begin{equation}\label{evol.intro}
\left\{
\begin{array}{ll}
\displaystyle \frac{\partial u_n}{\partial t} (x,t) = {L}_n u_n(x,t), \qquad & x\in \Omega, \, t>0, \\[10pt]
u_n(x,0)=u_0(x), \qquad & x\in \Omega.
\end{array}
\right.
\end{equation}
Notice that in \eqref{evol.intro} we have for $x \in B_n$,
\begin{equation}  \label{eq:main.Cauchy.local}
\displaystyle \frac{\partial u_n}{\partial t} (x,t)=\frac{1}{2}\Delta u_n (x,t) +
\int_{A_n} J(x,y)(u_n(y,t)-u_n(x,t))\, {d}y,\qquad  x\in B_n,\ t>0, 
\end{equation}
with zero Neumann boundary conditions on $\partial B_n$
(this condition is encoded in the domain of the generator ${L}_n$),
\begin{equation}  \label{eq:main.Cauchy.local.nn}
\displaystyle \frac{\partial u_n}{\partial \eta} (x,t) = 0, \qquad  x\in \partial B_n, \, t>0, 
\end{equation}
and a nonlocal equation for $x\in A_n$, 
\begin{equation}  \label{eq:main.Cauchy.nonlocal}
\frac{\partial u_n}{\partial t} (x,t)= \int_{A_n} J(x,y) (u_n(y,t)- u_n(x,t))\, {d}y +
 \int_{B_n} J(x,y)(u_n(y,t)-u_n(x,t))\, {d}y ,
\quad x \in A_n,\ t>0.
\end{equation}
Notice that in both parts of the equation there are coupling terms (terms involving
values in $A_n$ when $x\in B_n$ and integrals in $B_n$ when $x\in A_n$). These
coupling terms are nonlocal ones and come from the jumps from $A_n$ to $B_n$
and from $B_n$ to $A_n$.

In \cite{GQR}, using pure analysis of PDE methods,
it is proved that the Cauchy, Neumann and Dirichlet problems (in the last two cases with zero boundary data) 
for this evolution equation, \eqref{evol.intro}, with an integrable initial data $u_0$, are well posed in $L^p$ spaces. 
Moreover, the authors prove that the solutions to these problems share several properties with the solutions of the corresponding evolutions for their local and nonlocal counterparts~\eqref{heat.equation} and~\eqref{eq:nonlocal.equation}: there is conservation of the total mass, 
a comparison principle holds, and solutions converge to the mean value of the initial conditions as $t\to \infty$. We refer also to 
\cite{delia1,delia2,delia3,Du,Gal,Kra} for other results on coupling local and nonlocal evolution equations.

Our goal is to take the limit, as $n\to +\infty$, both in the processes $X_n(t)$ and in the associated densities $u_n(x,t)$. 
To this end we need to look at the process $X_n(t)$ as a couple $\pare{X_n\pare{t}, I_n\pare{t}}$. In our notation $I_n(t)$ contains explicitly the information over the set ($A_n$ or $B_n$) in which $X_n(t)$ is located. More precisely, $I_n\pare{t}=1$ (or $2$) if the particle is in $A_n$ (or in $B_n$ respectively)
at time $t$.

Before stating our main result we need to introduce some notation. Given a metric space $X$, for $T>0$, we denote by $D\pare{[0, T], X}$ the space of all trajectories cadlag defined in $[0, T]$ and taking values in $X$. We consider $D\pare{[0, T], X}$ endowed with the Skorohod topology (see Chapter 3 of \cite{bil} for more details). 
Our process $\pare{X_n\pare{t}, I_n\pare{t}}_{t\in [0, T]}$ is in $D\pare{[0, T], \overline{\Omega}}\times D\pare{[0, T],\brc{1,2}}$ which we consider endowed with the product topology. 

Now we are ready to state our main result that reads as follows.
\begin{theorem} \label{teo.main.intro}
Assume \eqref{cond.sets} and fix $T>0$. We have that, as $n\to \infty$,
\begin{equation}\label{limite.debil.u.intro}
\begin{array}{l}
\displaystyle u_n(x,t) \rightharpoonup u(x,t),\qquad \mbox{weakly in } L^2 (\Omega \times (0,T)), \\[5pt]
\displaystyle \chi_{B_n} (x) u_n(x,t) \rightharpoonup a(x,t),\qquad \mbox{weakly in } L^2 (\Omega \times (0,T)), \\[5pt]
\displaystyle \chi_{A_n} (x) u_n(x,t) \rightharpoonup b(x,t),\qquad \mbox{weakly in } L^2 (\Omega \times (0,T)). 
\end{array}
\end{equation}
These limits verify
$$
u(x,t) = a(x,t) + b(x,t)
$$
and are characterized by
the fact that $(a,b)$ is the unique solution to the following system,
\begin{equation}\label{sys1.intro}
\left\{
\begin{array}{ll}
\displaystyle \frac{\partial a}{\partial t}\pare{x, t}=\int_{\Omega}J\pare{x, y}\pare{a\pare{y,t}-a\pare{x,t}}\, dy
\\[10pt]
\displaystyle \qquad \qquad \quad 
-\theta\pare{x}\int_{\Omega}J\pare{x, y}a\pare{y,t}dy+\pare{1-\theta\pare{x}}\int_{\Omega}J\pare{x,y}b\pare{y,t}\, dy \quad & x\in \Omega , \, t>0,\\[10pt]
\displaystyle \frac{\partial b}{\partial t}\pare{x, t}=\theta\pare{x}\int_{\Omega} J\pare{x, y}a\pare{y,t}dy-b\pare{x,t}\int_{\Omega} 
J\pare{x, y}\pare{1-\theta (y)}\, dy
\quad & x\in \Omega , \, t>0,\\[10pt]
a\pare{x, 0}=\pare{1-\theta\pare{x}}u_0\pare{x}, \quad b\pare{x, 0}=\theta\pare{x}u_0\pare{x}
\qquad & x\in \Omega.
\end{array} \right.
\end{equation}

Moreover, it holds that
the sequence of processes converges in distribution 
\begin{align}\label{convd}
\pare{X_n\pare{t}, I_n\pare{t}}\xrightarrow[n\to +\infty]{D}\pare{X\pare{t}, I\pare{t}}
\end{align}
in $D\pare{[0, T], \overline{\Omega}}\times D\pare{[0, T], \brc{1,2}}$, where the distribution of the limit $\pare{X\pare{t}, I\pare{t}}$ is characterized by 
having as probability densities $a(x,t)$ and $b(x,t)$, that is,
\begin{align}
P\Big( X\pare{t}\in  E, I(t) =1 \Big) = \int_{E} a(z,t) \, dz \quad \mbox{and} \quad 
P \Big( X\pare{t}\in  E, I(t) =2 \Big) = \int_{E} b(z,t) \, dz ,
\end{align}
for every measurable set $ E\subseteq \overline\Omega$.
\end{theorem}

Notice that in this homogenization procedure we have two main features: in the limit the local part does not appear and
there is a nonlocal system instead of a single equation. The limit system can be
interpreted as follows: the particle that is moving according to the limit process keeps some memory
that there are two sets involved in the homogenization. Then we can look at the limit system \eqref{sys1.intro} as a system describing
the movement of a particle that has a label (white or black). The probability density of being white is given by $a(x,t)$ and 
the probability density of being black is $b(x,t)$. Remark that the total mass of the system $\int_\Omega a(x,t) dx +
\int_\Omega b(x,t) dx$ remains constant in time. The transition probabilities between labels and the jumps inside the domain
are encoded in system \eqref{sys1.intro}. For example, a particle that is black at $x$ at time $t$ becomes white and jumps 
to $y$ with a probability $\int_{\Omega} 
J\pare{x, y}\pare{1-\theta (y)}dy$ (this explains the last term in the two equations). Note that particles labeled black
does not jump and remain black at the same time (they can only jump if they change label and become white), 
while particles labeled white may also jump to
other positions keeping their label. This fact comes as a consequence that in the original process the jumps from $B_n$ to $B_n$ 
are suppressed.

Let us describe briefly the main ingredients that appear in the proofs. First, we show 
weak convergence along subsequences of $u_n$, $a_n$ and $b_n$ (these convergences comes from
a uniform bound in $L^2$). Next, we find the system that these limits verify, \eqref{sys1.intro}. This part
is delicate since we need to use an approximation lemma that provides us with test
functions whose laplacians vanish on $B_n$ (it is here 
where we use that the diameter of the components of $B_n$ goes to zero). In addition,  
passing to the limit in a term like $\chi_{B_n} (x) \chi_{A_n} (y) J(x,y)$
constitutes one of the main issues of our proofs since here we have only weak convergence of 
$\chi_{B_n} $ and $\chi_{A_n} $. Here we need to rely on the continuity of $J$ and use the fact that the product
$\chi_{B_n} (x) \chi_{A_n} (y) J(x,y)$ involves two different variables, $x$ and $y$. Finally, we show
uniqueness of the limit by proving uniqueness of solutions to the limit system \eqref{sys1.intro}.
To show convergence of the densities we rely on pure analysis methods (no probabilistic arguments
are needed here).

The passage to the limit in the processes $X_n$ is delicate since the limit of the corresponding densities
is given by a system. Hence, we need to consider here the pair $\pare{X_n , I_n}$ with 
the extra variable that takes into account when the particle is in $A_n$ or in $B_n$ and then prove
tightness of the pair to obtain a limit in distribution. To characterize this limit and relate it to the limit
of the densities is the final step of the proof. Here we use probabilistic arguments together with similar ideas to the ones used
in passing to the limit in the densities to obtain the convergence of the different terms
that appear. Here, we need to show uniqueness of evolution problems in the space
of measures to show, first, that $u_n(x,t)$ is the density of $X_n(t)$ and, finally,
that $(a(x,t),b(x,t))$ are the densities associated to the limit process $(X(t),I(t))$.

The evolution problem \eqref{evol.intro} is the gradient flow in $L^2(\Omega)$ of the energy
$$
E_n (u) = \frac14 \int_{B_n} |\nabla u|^2 + \frac14 \int_{A_n} \int_{A_n}
J(x,y) |u(y) - u(x)|^2 dy dx + \frac12 \int_{A_n} \int_{B_n}
J(x,y) |u(y) - u(x)|^2 dy dx.
$$
Therefore, one may be tempted to use convergence of the energies to obtain 
convergence of the densities (using, for example, Mosco convergence, see \cite{Mosco} and \cite{Brezis-Pazy}).
We are not following this approach here since in the limit we have a system for which
it is not clear how a limit energy functional looks like (if there is any).

Let us also briefly comment on the main hypothesis, condition \eqref{cond.sets}.
As we mentioned before, the fact that the limit $\theta (x)$ is assumed to be strictly between $0$ and $1$
reflects the idea that the local and nonlocal regions are mixed in the whole 
reference domain $\Omega$. The condition that says that the diameter of the
connected components of $B_n$ goes to zero is crucial for our result since it implies
that the Brownian motion part of the evolution gets trapped in very small sets as $n\to \infty$
(recall that the Brownian motion is supplemented with a reflexion on the boundary of
$B_n$).
In the final section we will present an example (in a square we just take narrow strips parallel to one of the sides
as the partition $A_n$, $B_n$) that shows that this condition is necessary
in order to obtain our main result, Theorem \ref{teo.main.intro}.

Homogenization for PDEs is by now a classical subject that originated in the study of the behaviour
of the solutions to elliptic and parabolic local equations with highly oscillatory coefficients (periodic homogenization). 
We refer to
\cite{Ben,TT} as general references for the subject.
For other kind of homogenization (using PDEs techniques without mixing local and nonlocal processes) for pure nonlocal problems we
refer to \cite{marc1,marc2,marc3}. For homogenization results for singular 
kernels (but without the local part) we refer to \cite{Ca,sw2,Wa} and references therein. We emphasize that those references
deal with homogenization in the coefficients involved in the equation. For random homogenization of an obstacle problem we refer to \cite{caffa2}. 
Here we deal with an homogenization problem that is different in nature with the ones treated in the previously mentioned references
as we homogenize mixing two operators/processes that are different in nature (local/nonlocal).

The paper is organized as follows: In Section \ref{sect-lemma} we include a technical lemma
that will be the key to obtain that the Laplacian disappears in the limit. In Section \ref{sec-teo-main} 
we prove Theorem \ref{teo.main.intro} (first we deal with the limit of the densities and next we
compute the limit of the processes); finally,
in Section \ref{sect-generaliz} we include some examples, including the 1-d case  $\overline\Omega=[0,1]$,
a chessboard configuration in $\overline\Omega = [0,1] \times [0,1]$ and a thin strips configuration in 
$\overline\Omega = [0,1] \times [0,1]$
that shows that the condition involving the diameter on the connected components of $B_n$ is needed in order to obtain
a limit in which the local part disappears in the limit.

\section{An approximation lemma} \label{sect-lemma}

In this section we prove a technical lemma that helps us to modify a function $\phi (x,t)$ in order to approximate it with
functions $\phi_n(x,t)$ that are constant in each of the components of $B_n$, and hence we have
$$
\chi_{B_n}\pare{x} \Delta \phi_n \pare{x,t} \equiv 0.
$$
This approximation procedure will be used to prove that the local part of the process disappears in the homogenization limit.
We will use that, for every $n\in \bb N$ and $x\in \overline\Omega$, 
we have that $x\in A_n$ or there exists $j\in\brc{1,\cdots, N}$ with $x\in B_n^j$. 
Here and in what follows we will use the notation $C^{1}\pare{[0, T], C(\overline{\Omega})}$ to denote the space
of functions that are continuously differentiable in time and continuous in space; also $L^p \pare{0, T : L^q({\Omega})}$ denotes the space of functions that 
are in $L^p$ in time with values in $L^q(\Omega)$.

\begin{lemma} \label{lema-phi}
Given $\phi :[0,T] \times \overline\Omega \mapsto \mathbb{R}$, let
\begin{equation} \label{def.clave}
\phi_n\pare{x, t}=
\left\{
\begin{array}{ll} \displaystyle 
\kint_{B_n^j}\phi(z, t)dz &\text{if }x\in B_n^j,\; t\geq 0,\\[10pt]
\phi\pare{x, t} &\text{if }x\in A_n,\; t\geq 0.
\end{array}
\right.
\end{equation}

\noindent
If $\phi\pare{x, t}\in C^{1}\pare{[0, T], C(\overline{\Omega})}$, then 
$$
\phi_n (\cdot,t) \in C(A_n) \cap C^2 (B_n),
$$
and it holds that, 
$$\phi_n\xrightarrow[n\to +\infty]{}\phi, \qquad \mbox{and} \qquad 
\frac{\partial \phi_n}{\partial t} \xrightarrow[n\to +\infty]{} \frac{\partial \phi}{\partial t}, \qquad \mbox{ uniformly in } [0, T]\times \overline{\Omega}.$$
Therefore,
$$\phi_n\xrightarrow[n\to +\infty]{}\phi \qquad \mbox{and} 
\qquad \frac{\partial \phi_n}{\partial t}  \xrightarrow[n\to +\infty]{}\frac{\partial \phi}{\partial t}, \qquad \mbox{strongly in } L^2\pare{0, T:L^2\pare{\Omega}}.
$$
\end{lemma}

\begin{proof}
Our goal is to prove that 
\begin{align}\label{ucv}
\norm{\phi-\phi_n}_{L^\infty ([0,T] \times \overline{\Omega})} \xrightarrow[n\to +\infty]{}0.
\end{align}

First of all, observe that, since the function $\phi$ is uniformly continuous in $[0, T] \times\overline{\Omega}$,
for every $\epsilon>0$ there exists $\delta_{\epsilon}$ such that 
\begin{align}\label{acf}
\abs{\phi(x, t)-\phi(y, t)}<\epsilon, \qquad \forall t\geq 0,\; \forall x,y :\abs{x-y}<\delta_\epsilon.
\end{align}
Fix $\epsilon>0$, $t\geq 0$ and take $x\in B_n$, consequently there exists $j\in\brc{1, \ldots, N}$ such that $x\in B_n^j$. Therefore, we have
\begin{equation}\label{pl1.kk}
\abs{\phi_n\pare{x, t}-\phi\pare{x, t}} \displaystyle =\abs{\kint_{B^n_j}\phi(z, t)dz-\phi(x, t)}
\displaystyle \leq \kint_{B^n_j}\abs{\phi(z, t)-\phi(x, t)}dz
\leq \epsilon,
\end{equation}
for every $n$ such that $\max_j\brc{\mbox{diam} (B_n^j)} < \delta_{\epsilon}$  (here we are using our 
hypothesis \eqref{cond.sets}). 
While if $x\in  A_n$, we have
\begin{align}\label{pl2.kk}
\abs{\phi_n\pare{x, t}-\phi \pare{x, t}} =0,
\end{align}
for all $n\in\bb N$.

From \eqref{pl1.kk} and \eqref{pl2.kk} we get that there exists $n_0\pare{\epsilon}$ such that $\forall n\geq n_0\pare{\epsilon}$ 
\begin{align}\label{moni}
\norm{\phi_n-\phi}_{L^\infty ([0,T] \times \overline{\Omega})}<\epsilon
\end{align}
and this allows to conclude our first statement. The uniform convergence of the time derivatives can be proved analogously just by observing that
$$
\frac{\partial \phi_n}{\partial t}\pare{x, t}=
\left\{
\begin{array}{ll} \displaystyle 
\kint_{B_n^j} \frac{\partial \phi}{\partial t} (z, t)dz &\text{if }x\in B_n^j, \, t\geq 0\\[10pt]
\phi_t \pare{x, t} &\text{if }x\in A_n, \, t\geq 0,
\end{array}
\right.
$$
and then the same arguments as before can be applied.
\end{proof}

\begin{remark} {\rm By definition the sequence of functions $\pare{\phi_n}_{n\in\bb N}$ is such that 
$$\sup_{n\in \bb N}\norm{\phi_n}_{L^\infty ([0,T]\times \overline{\Omega})} \leq \norm{\phi}_{L^\infty ([0,T]\times \overline{\Omega})}.$$
}
\end{remark}

\begin{remark} \label{rem.cont}{\rm If we assume, in addition to our previous conditions on the configurations
$\overline\Omega = A_n \cup B_n$, \eqref{cond.sets}, that the connected components of $B_n$ are strictly separated, that is,
for every $n$, there exists $\delta_n>0$ such that,
\begin{equation} \label{extra}
\inf_{i\neq j} \mbox{dist} (B^n_j, B^n_i) = \delta_n >0,
\end{equation}
then we can modify our approximations in such a way that $$\phi_n\in C(\overline{\Omega}).$$
In fact, let us briefly sketch the arguments. First, we enlarge a little the set $B_n$, taking $\widetilde{B_n} = B_n + B_{\eta_n}(0)$ with $\eta_n \ll \delta_n$.
Notice that, from the extra condition \eqref{extra}, we have the same number of 
connected components in $\widetilde{B_n}$ (that we call $\widetilde{B_n}^j$) and in $B_n$.
Then, we let
\begin{equation} \label{def.clave.99}
\psi_n\pare{x, t}=
\left\{
\begin{array}{ll} \displaystyle 
\kint_{\widetilde{B_n}^j}\phi(z, t)dz &\text{if }x\in \widetilde{B_n}^j, t\geq 0,\\[10pt]
\phi\pare{x, t} &\text{if }x\in A_n \setminus \left( \cup_j \widetilde{B_n}^j \right), \  t\geq 0.
\end{array}
\right.
\end{equation}
This function $\psi_n$ is not necessarily continuous but, from our previous arguments, is uniformly close to $\phi$.
Finally, we just take our approximating sequence to be
$$
\phi_n (x,t) = \psi_n  \ast_x \rho_{\epsilon_n} (x,t)
$$
being $\rho_{\epsilon_n}(x)$ a smooth mollifier with $\epsilon_n\ll \eta_n$.
}
\end{remark}

\section{Proof of Theorem \ref{teo.main.intro}} \label{sec-teo-main}

\subsection{Convergence of the densities} Now we look for the limit of the solution to \eqref{evol.intro}.
That is, our goal is to pass to the limit in $\{u_n\pare{x, t}\}_n$ being $u_n$ the solution of the following evolution problem:
\begin{align}\label{sys1}
\begin{cases}
\displaystyle \frac{\partial u_n}{\partial t} \pare{x, t}=&
\displaystyle \chi_{A_n}\pare{x}\int_{\Omega}J\pare{x,y}\pare{u_n\pare{y, t}-u_n\pare{x, t}}dy\\[10pt]
& \displaystyle  +\chi_{B_n}\pare{x}\int_{\Omega} \chi_{A_n}\pare{y}J\pare{x,y}\pare{u_n\pare{y, t}-u_n\pare{x, t}}dy  +\frac{1}{2}\chi_{B_n}\pare{x} \Delta \pare{u_n} \pare{x, t}, \quad x\in \Omega, \, t>0,\\[10pt]
& \hspace{-68pt} \displaystyle \frac{\partial u_n}{\partial \eta} \pare{x, t}=0, \qquad x \in \partial B_n, \, t>0,\\[10pt]
&\hspace{-68pt}u_n\pare{x, 0}=u_0\pare{x}, \qquad x \in \Omega.
\end{cases}
\end{align}

It is proved in \cite{GQR} that there exists a unique solution $u_n\pare{x, t}$ of system \eqref{sys1}. 
The following lemma shows that the sequence $\{u_n\}_n$ is uniformly bounded.

\begin{lemma} \label{lema-3-2-}
There exists a constant $C$ such that
\begin{align}
\norm{u_n}_{L^\infty \pare{0, T: L^2(\Omega)}}\leq C.
\end{align}
\end{lemma}

\begin{proof}
To prove the uniform bound we just multiply by $u_n$ both sides of \eqref{sys1} and integrate in $\Omega$ and in $[0,T]$ to obtain
\begin{align}\label{wq}
\begin{array}{l}
\displaystyle 
\frac12\int_\Omega (u_n)^2 (x,T) dx -  
\frac12\int_\Omega (u_0)^2 (x) dx = \int_0^T \int_\Omega   {L}_n u_n(x,t) u_n (x,t) dx dt  \\[10pt]
\qquad \displaystyle 
= \int_0^T \int_\Omega \chi_{A_n}\pare{x}\int_\Omega J\pare{x,y}\pare{u_n\pare{y,t}-u_n\pare{x,t}}dy u_n (x,t) 
dx dt 
\\[10pt]
\qquad \displaystyle 
\qquad + \int_0^T \int_\Omega \chi_{B_n}\pare{x}\int_\Omega \chi_{A_n}\pare{y}J\pare{x,y}\pare{u_n\pare{y,t}- u_n\pare{x,t}}dy
u_n (x,t) 
dx dt 
\\[10pt]
\qquad \displaystyle 
\qquad
+ \int_0^T \int_\Omega \chi_{B_n}\pare{x} \Delta u_n \pare{x,t }  u_n (x,t) 
dx dt 
\\[10pt]
\qquad \displaystyle 
\leq  - \int_0^T \iint_{\Omega \times \Omega} (1- \chi_{B_n}\pare{x}
\chi_{B_n}\pare{y} ) J(x,y) \pare{u_n\pare{y,t}-u_n\pare{x,t}}^2 dx dy dt \leq 0,
\end{array}
\end{align}
and hence the $L^2-$norm of the solution is decreasing in time and the result follows. To get the last step of \eqref{wq} we used that the function $J$ is symmetric and that
\begin{align}
\int_0^T \int_\Omega \chi_{B_n}\pare{x} \Delta u_n \pare{x,t }  u_n (x,t) 
dx dt =-\int_0^T \int_\Omega \chi_{B_n}\pare{x} | \nabla u_n \pare{x,t }|^2
dx dt \leq 0
\end{align}
that holds due to the fact that $\frac{\partial u_n}{\partial \eta} \pare{x, t}=0$ for $x \in \partial B_n, \, t>0$.
\end{proof}

Now, we let 
$$a_n\pare{x, t}=\chi_{A_n}u_n\pare{x,t}\qquad \mbox{ and }\qquad b_n\pare{x,t}=\chi_{B_n}u_n\pare{x,t}.$$
Since $\overline\Omega = A_n \cup B_n$ we have
$$
u_n (x,t)= \chi_{A_n} (x) u_n(x,t) + \chi_{B_n} (x) u_n(x,t) = a_n (x,t)+b_n (x,t).
$$ 

Our next task is to show that there is a limit, as $n\to \infty$ (along a subsequence), that is a solution to the limit system 
\eqref{sys1.intro}.

\begin{theorem}\label{th32}
It holds that $$u_n\xrightarrow[n\to +\infty]{}u, \qquad a_n\xrightarrow[n\to +\infty]{}a \qquad \mbox{ and  } \qquad
b_n\xrightarrow[n\to +\infty]{}b$$ weakly in $L^2\pare{\Omega \times (0,T)}$ where these limits verify
$$
u(x,t) = a(x,t) + b(x,t),
$$
and the pair $\pare{a\pare{x, t}, b\pare{x, t}}$ is a solution to the system
\eqref{sys1.intro}.
\end{theorem}

\begin{proof}
From Lemma \ref{lema-3-2-} we have that 
\begin{align}
\norm{u_n}_{L^\infty \pare{0, T: L^2(\Omega)}};\  \norm{a_n}_{L^\infty \pare{0, T: L^2(\Omega)}};\  \norm{b_n}_{L^\infty \pare{0, T: L^2(\Omega)}}\leq C.
\end{align}
Therefore, the sequences $\{u_n\}_n$, $\{a_n\}_n$ and $\{b_n\}_n$ are bounded in $L^2$ uniformly in $n$ and then 
we can extract a weakly convergent subsequence of $\{u_n\}_n$, $\{a_n\}_n$ and $\{b_n\}_n$ that for simplicity of notation we index again by $n$. 
We call $u$, $a$ and $b$ the weak limits of the subsequences $\{u_n\}_n$, $\{a_n\}_n$ and $\{b_n\}_n$ respectively.
From 
$$
u_n(x,t)  =  a_n(x,t) + b_n(x,t),
$$
we immediately get
$$
u(x,t) = a(x,t) + b(x,t).
$$

Take a smooth function $\phi$ such that $\phi(\cdot, T)\equiv 0$ and approximate it with
functions $\phi_n$ that are constant in each of the components of $B_n$ as was shown in Lemma \ref{lema-phi}.

Consider now equation \eqref{sys1} and multiply both sides by $\chi_{B_n}\pare{x}\phi_n\pare{x, t}$ and then integrate respect to the variables $x$
and $t$. Since by construction $\phi_n(\cdot, T)\equiv 0$, integrating by parts we obtain
\begin{align}
-\int_0^T\int_\Omega &\frac{\partial \phi_n}{\partial t}(x, t)b_n\pare{x, t}\, dxdt+\int_\Omega \chi_{B_n}\pare{x}u_0\pare{x}\phi_n\pare{x,0}\, dx\\[10pt]
&=\int_0^T\int_\Omega \int_\Omega \chi_{B_n}\pare{x}\chi_{A_n}\pare{y}J\pare{x,y}\pare{u_n\pare{y, t}-u_n\pare{x, t}}\phi_n\pare{x, t}\, dydxdt
\\[10pt]
&\hspace{+15pt}+\frac{1}{2}\int_0^T\int_\Omega \chi_{B_n}\pare{x}\Delta  u_n\pare{x, t} \phi_n\pare{x, t}\, dx dt\\[10pt]
&=\int_0^T\int_\Omega \int_\Omega J\pare{x, y}\chi_{B_n}\pare{x}a_n\pare{y, t}\phi_n\pare{x, t}\, dydxdt\\[10pt]
&\hspace{+15pt}-\int_0^T\int_\Omega \int_\Omega J\pare{x, y}\chi_{A_n}\pare{y}b_n\pare{x, t}\phi_n\pare{x, t}\, dydxdt.
\end{align}
Since $$\chi_{A_n}\pare{\cdot}\xrightarrow[n\to +\infty]{}1-\theta\pare{\cdot} \qquad \mbox{ and } \qquad \chi_{B_n}\pare{\cdot}\xrightarrow[n\to +\infty]{}\theta\pare{\cdot}$$ weakly in $L^2\pare{\Omega \times (0,T)}$, we obtain the following limits (here we use the strong convergences proved in Lemma \ref{lema-phi})
$$
\int_0^T\int_\Omega \frac{\partial \phi_n}{\partial t} \pare{x, t}b_n\pare{x, t}\, dxdt
\xrightarrow[n\to +\infty]{}\int_0^T\int_\Omega \frac{\partial \phi}{\partial t} \pare{x, t}b\pare{x, t}\, dxdt,
$$
$$
\int_\Omega \chi_{B_n}\pare{x}u_0\pare{x}\phi_n\pare{x,0}\, dx\xrightarrow[n\to +\infty]{}\int_0^1\theta\pare{x}u_0(x)\phi\pare{x,0}\, dx.
$$
Now, as we assumed that $J(x,y)$ is continuous, we have that
$$
h_n(y) = \int_\Omega  J(x,y) \chi_{B_n} (x) \phi_n\pare{x, t} \, dx   
\xrightarrow[n\to +\infty]{} \int_\Omega  J(x,y) \theta (x) \phi \pare{x, t} \, dx 
$$
uniformly in $y$. Therefore, we get
$$
\int_0^T\int_\Omega \int_\Omega J\pare{x, y}\chi_{B_n}\pare{x}a_n\pare{y, t}\phi_n\pare{x, t}\, dydxdt\xrightarrow[n\to +\infty]{}\int_0^T\int_\Omega 
\int_\Omega J\pare{x, y}\theta\pare{x}a\pare{y, t}\phi\pare{x, t}\, dydxdt,
$$
and, arguing similarly,
$$
\begin{array}{l}
\displaystyle
\int_0^T\int_\Omega \int_\Omega J\pare{x, y}\chi_{A_n}\pare{y}b_n\pare{x, t}\phi_n\pare{x, t}\, dydxdt 
\xrightarrow[n\to +\infty]{}
\int_0^T\int_\Omega  \int_\Omega J\pare{x, y}\pare{1-\theta\pare{y}}b\pare{x, t}\phi\pare{x, t}\, dydxdt.
\end{array}
$$

Collecting all these limits we conclude that
\begin{align}
-\int_0^T\int_\Omega & \frac{\partial{\phi}}{\partial t} \pare{x, t}b \pare{x, t}dxdt+\int_\Omega \theta \pare{x}u_0\pare{x}\phi \pare{x,0}\, dx\\[10pt]
&=\int_0^T\int_\Omega \int_\Omega J\pare{x, y}\theta\pare{x}a \pare{y, t}\phi \pare{x, t}\, dydxdt\
-\int_0^T\int_\Omega \int_\Omega J\pare{x, y}(1-\theta\pare{y})b\pare{x, t}\phi\pare{x, t}\, dydxdt.
\end{align}
Since this holds for every $\phi$, we conclude that $b\pare{x, t}$ is a solution to
\begin{align}
\begin{cases}
& \displaystyle
\frac{\partial b}{\partial t}\pare{x, t}=\theta\pare{x}\int_\Omega J\pare{x, y}a\pare{y,t}dy-b\pare{x,t}\int_\Omega J\pare{x, y}\pare{1-\theta\pare{y}}dy\\[10pt]
&b\pare{x, 0}=\theta\pare{x}u_0\pare{x} .
\end{cases}
\end{align}

Following a similar strategy, multiply now both sides of equation \eqref{sys1} by $\chi_{A_n}\pare{x}\phi_n\pare{x, t}$ and then integrate respect to the variables $x$ and $t$ to obtain 
\begin{align}\label{sysk2} 
-\int_0^T\int_\Omega &\frac{\partial \phi_n}{\partial t} \pare{x, t}a_n\pare{x, t}\, dxdt+\int_\Omega \chi_{A_n}\pare{x}u_0\pare{x}\phi_n\pare{x,0}\, dx\\[10pt]
&=\int_0^T\int_\Omega \int_\Omega \chi_{A_n}\pare{x}J\pare{x,y}\pare{u_n\pare{y, t}-u_n\pare{x, t}}\phi_n\pare{x, t}\, dydxdt\\[10pt]
&=\int_0^T\int_\Omega \int_\Omega J\pare{x, y}\chi_{A_n}\pare{x}\pare{a_n\pare{y, t}+b_n\pare{y,t}}\phi_n\pare{x, t}\, dydxdt\\[10pt]
&\hspace{+15pt}-\int_0^T\int_\Omega \int_\Omega J\pare{x, y}a_n\pare{x, t}\phi_n\pare{x, t}\, dydxdt.
\end{align}

Arguing as before we can conclude that
$$
\int_0^T\int_\Omega \frac{\partial \phi_n}{\partial t} \pare{x, t}a_n\pare{x, t}\, dxdt\xrightarrow[n\to +\infty]{}\int_0^T\int_\Omega 
\frac{\partial \phi}{\partial t} \pare{x, t}a\pare{x, t}\, dxdt,
$$
$$
\int_\Omega \chi_{A_n}\pare{x}u_0\pare{x}\phi_n\pare{x,0}\, dx\xrightarrow[n\to +\infty]{}\int_\Omega \pare{1-\theta\pare{x}}u_0(x)\phi\pare{x,0}\, dx,
$$
$$
\begin{array}{l}
\displaystyle
\int_0^T\int_\Omega \int_\Omega J\pare{x, y}\chi_{A_n}\pare{x}\pare{a_n\pare{y, t}+b_n\pare{y,t}}\phi_n\pare{x, t}\, dydxdt\\[10pt]
\displaystyle \hspace{+120pt}\xrightarrow[n\to +\infty]{}\int_0^T\int_\Omega \int_\Omega 
J\pare{x, y}\pare{1-\theta\pare{x}}\pare{a\pare{y, t}+b\pare{y,t}}\phi\pare{x, t}\, dydxdt,
\end{array}
$$
and
$$
\int_0^T\int_\Omega \int_\Omega J\pare{x, y}a_n\pare{x, t}\phi_n\pare{x, t}\, dydxdt\xrightarrow[n\to +\infty]{}\int_0^T\int_\Omega 
\int_\Omega J\pare{x, y}a\pare{x, t}\phi\pare{x, t}\, dydxdt.
$$

Therefore, we get
\begin{align}\label{sys2}
-\int_0^T\int_\Omega & \frac{\partial \phi}{\partial t} \pare{x, t}a \pare{x, t}dxdt+\int_\Omega (1-\theta\pare{x})u_0\pare{x}\phi\pare{x,0}dx\\[10pt]
&=\int_0^T\int_\Omega \int_\Omega J\pare{x, y} (1-\theta\pare{x})\pare{a\pare{y, t}+b\pare{y,t}}\phi\pare{x, t}dydxdt
-\int_0^T\int_\Omega \int_\Omega J\pare{x, y}a\pare{x, t}\phi\pare{x, t}dydxdt,
\end{align}
and hence we obtain that $a\pare{x, t}$ is a solution to
\begin{align}
\begin{cases}
&\displaystyle \frac{\partial a}{\partial t}\pare{x, t}=\int_\Omega 
J\pare{x, y}\pare{a\pare{y,t}-a\pare{x,t}}dy-\theta\pare{x}\int_\Omega 
J\pare{x, y}a\pare{y,t}dy\displaystyle +\pare{1-\theta\pare{x}}\int_\Omega J\pare{x,y}b\pare{y,t}dy\\[10pt]
&a\pare{x, 0}=\pare{1-\theta\pare{x}}u_0\pare{x}.
\end{cases}
\end{align}
The proof is finished.
\end{proof}
Now,  to obtain convergence of the whole sequences $\{a_n\}_n$, $\{b_n\}_n$ (not only along subsequences) 
we need to prove that the limit system admits a unique solution. 
In fact uniqueness holds for measure valued solutions as will we show in Lemma \ref{lemma:weaksys}
but here we need only the content of the following lemma to obtain uniqueness of the limit.

\begin{lemma}\label{lemma:eu}
There exists a unique solution to the system 
\begin{equation}\label{sys1.intro.99}
\left\{
\begin{array}{ll}
\displaystyle \frac{\partial a}{\partial t}\pare{x, t}=\int_{\Omega}J\pare{x, y}\pare{a\pare{y,t}-a\pare{x,t}}\, dy
\\[10pt]
\displaystyle \qquad \qquad \quad 
-\theta\pare{x}\int_{\Omega}J\pare{x, y}a\pare{y,t}dy+\pare{1-\theta\pare{x}}\int_{\Omega}J\pare{x,y}b\pare{y,t}\, dy \quad & x\in \Omega , \, t>0,\\[10pt]
\displaystyle \frac{\partial b}{\partial t}\pare{x, t}=\theta\pare{x}\int_{\Omega} J\pare{x, y}a\pare{y,t}dy-b\pare{x,t}\int_{\Omega} 
J\pare{x, y}\pare{1-\theta (y)}\, dy
\quad & x\in \Omega , \, t>0,\\[10pt]
a\pare{x, 0}=\pare{1-\theta\pare{x}}u_0\pare{x}, \quad b\pare{x, 0}=\theta\pare{x}u_0\pare{x}
\qquad & x\in \Omega.
\end{array} \right.
\end{equation}
\end{lemma}
\begin{proof} First, we observe that we have existence of a solution from our previous
limit.

Let us prove uniqueness.
Suppose that $\pare{a_1\pare{x, t}, b_1\pare{x,t}}$ and $\pare{a_2\pare{x, t}, b_2\pare{x,t}}$ are two solutions of equation \eqref{sys1.intro}. Call $a\pare{x, t}=a_1\pare{x, t}-a_2\pare{x, t}$ and $b\pare{x, t}=b_1\pare{x, t}-b_2\pare{x, t}$. Then $\pare{a\pare{x,t}, b\pare{x, t}}$ satisfies the same system \eqref{sys1.intro} but with
$$
\displaystyle  a \pare{x, 0}\equiv 0, \qquad \mbox{ and } \qquad b\pare{x, 0}\equiv 0,
$$
as initial conditions. Our goal is to show that $a\equiv b \equiv 0$.

From the first equation in system \eqref{sys1.intro}, multiplying by $a(x,t)$ and integrating in $x$, we get
\begin{align}
\frac{\partial}{\partial t}\int_\Omega \frac{a^2\pare{x, t}}{2}dx=&\int_\Omega a\pare{x, t}\pare{\int_\Omega J\pare{x, y}{a\pare{y,t}dy}}dx
-\int_\Omega 
a^2\pare{x,t}\pare{\int_\Omega J\pare{x, y}dy}dx\\[10pt]
&-\int_\Omega a\pare{x,t}\theta\pare{x}\pare{\int_\Omega J\pare{x, y}a\pare{y,t}dy}dx
+\int_\Omega a\pare{x,t}\pare{1-\theta\pare{x}}\pare{\int_\Omega 
J\pare{x,y}b\pare{y,t}dy}dx.
\end{align}
Observe that from Cauchy-Schwarz's inequality we get that
\begin{align}
\abs{\int_\Omega a\pare{x, t}\pare{\int_\Omega J\pare{x, y}{a\pare{y,t}dy}}dx}&\leq \cor{\int_\Omega a^2\pare{x, t}dx}^{\frac{1}{2}} \cor{\int_\Omega \pare{\int_\Omega J\pare{x, y}a\pare{y,t}dy }^2dx}^{\frac{1}{2}}\\[10pt]
&\leq {\int_\Omega a^2\pare{x, t}dx}+ {\int_\Omega \pare{\int_\Omega J\pare{x, y}a\pare{y,t}dy }^2dx}\\[10pt]
&\leq \int_\Omega a^2\pare{x, t}dx+\int_\Omega \int_\Omega J^2(x,y)dxdy\int_\Omega\int_\Omega a^2(y, t)dydx\\[10pt]
&\leq\pare{ 1+ \abs{\Omega}^2\norm{J}_\infty^2}\int_\Omega a^2\pare{x, t}dx.
\end{align}
Since $\int_\Omega J\pare{x, y}dy\equiv 1$ we also have the following bounds
\begin{align}
\abs{\int_\Omega a^2\pare{x,t}\pare{\int_\Omega J\pare{x, y}dy}dx}\leq \int_\Omega a^2\pare{x, t}dx,
\end{align}
\begin{align}
\abs{\int_\Omega a\pare{x,t}\theta\pare{x}\pare{\int_\Omega J\pare{x, y}a\pare{y,t}dy}dx}&\leq \cor{\int_\Omega
a^2(x, t)\theta^2\pare{x}dx}^{\frac{1}{2}}\cor{\int_\Omega \pare{\int_\Omega J(x,y)a(y, t)dy}^2dx}^{\frac{1}{2}}\\[10pt]
&\leq\int_\Omega a^2(x, t)dx+\int_\Omega \int_\Omega J^2(x,y)dxdy\int_\Omega\int_\Omega a^2(y, t)dydx\\[10pt]
&\leq\pare{1+\abs{\Omega}^2\norm{J}_\infty^2}\int_\Omega a^2\pare{x, t}dx
\end{align}
and 
\begin{align}
 \abs{\int_\Omega a\pare{x,t}\pare{1-\theta\pare{x}}\pare{\int_\Omega
J\pare{x,y}b\pare{y,t}dy}dx}
& \leq \cor{\int_\Omega a^2(x, t)\pare{1-\theta\pare{x}}^2dx}^{\frac{1}{2}}\cor{\int_\Omega
\pare{\int_\Omega J(x,y)b(y, t)dy}^2dx}^{\frac{1}{2}}\\[10pt]
&  \leq\int_\Omega a^2(x, t)dx+\int_\Omega \int_\Omega J^2(x,y)dxdy\int_\Omega\int_\Omega b^2(y, t)dydx\\[10pt]
& \leq \int_\Omega a^2\pare{x, t}dx+\abs{\Omega}^2\norm{J}_\infty^2\int_\Omega b^2\pare{x, t}dx.
\end{align}

Therefore, we have that there exists a constant $C_1=4\pare{1+\abs{\Omega}^2\norm{J}_\infty^2}$ such that
\begin{equation} \label{eq.11}
\frac{\partial}{\partial t}\int_\Omega \frac{a^2\pare{x, t}}{2}dx \leq C_1\left( \int_\Omega a^2\pare{x, t}dx+\int_\Omega b^2\pare{x, t}dx \right).
\end{equation}

Now, from the second equation in system \eqref{sys1.intro} we obtain
\begin{align}
\frac{\partial}{\partial t}\int_\Omega 
\frac{b^2\pare{x, t}}{2}dx=&\int_\Omega b\pare{x, t}\theta\pare{x}\pare{\int_\Omega
J\pare{x, y}{a\pare{y,t}dy}}dx  -\int_\Omega 
b^2\pare{x,t}\pare{\int_\Omega 
J\pare{x, y}\pare{1-\theta\pare{y}}dy}dx.
\end{align}
Using again Cauchy-Schwarz's inequality we get that
\begin{align}
\abs{\int_\Omega b\pare{x, t}\theta\pare{x}\pare{\int_\Omega
J\pare{x, y}{a\pare{y,t}dy}}dx}
& \leq \cor{\int_\Omega b^2\pare{x, t}\theta^2\pare{x}dx}^{\frac{1}{2}} \cor{\int_\Omega
\pare{\int_\Omega J\pare{x, y}a\pare{y,t}dy }^2dx}^{\frac{1}{2}}\\[10pt]
&  \leq {\int_\Omega b^2\pare{x, t}dx}+ {\int_\Omega \pare{\int_\Omega J\pare{x, y}a\pare{y,t}dy }^2dx}\\[10pt]
&  \leq \int_\Omega b^2\pare{x, t}dx+\int_\Omega \int_\Omega J^2(x,y)dxdy\int_\Omega\int_\Omega  a^2(y, t)dydx\\[10pt]
& \leq\int_\Omega b^2\pare{x, t}dx+\abs{\Omega}^2{\norm{J}_\infty^2}\int_\Omega a^2\pare{x, t}dx
\end{align}
and
\begin{align}
\int_\Omega b^2\pare{x,t}\pare{\int_\Omega J\pare{x, y}\pare{1-\theta\pare{y}}dy}dx\leq \abs{\Omega}{\norm{J}_\infty}\int_\Omega b^2\pare{x, t}dx.
\end{align}
Therefore
\begin{equation} \label{eq.22}
\frac{\partial}{\partial t}\int_\Omega \frac{b^2\pare{x, t}}{2}dx \leq C_2\left( \int_\Omega a^2\pare{x, t}dx+\int_\Omega b^2\pare{x, t}dx \right),
\end{equation}
where $C_2=1+\abs{\Omega}^2\norm{J}^2_\infty+\abs{\Omega}\norm{J}_\infty$. 
\noindent
By \eqref{eq.11} and \eqref{eq.22} we have
\begin{align}
\frac{\partial}{\partial t}\int_\Omega \frac{a^2\pare{x, t}}{2}dx+\frac{\partial}{\partial t}\int_\Omega
\frac{b^2\pare{x, t}}{2}dx\leq 2\max\brc{C_1, C_2}\pare{\int_\Omega a^2\pare{x, t}dx+\int_\Omega b^2\pare{x, t}dx}.
\end{align}
Since $a(x, 0)\equiv b(x, 0)\equiv 0$, by Gronwell's inequality we can obtain that $$a\pare{x, t}\equiv b\pare{x, t}\equiv 0.$$ 
This concludes the proof.
\end{proof}

\subsection{Convergence of the processes} \label{subsect-proc}

In this section we analyze the convergence in distribution of the process $\pare{X_n(t), I_n(t)}_{t\in [0,T]}$ in a bounded interval of time $[0, T]$.
We already explained in the introduction, Section \ref{sec.intro}, that $\pare{X_n(t)}_t$ is a Markov process with generator ${L}_n$ defined by \eqref{gen.intro}.

Our first goal is to show that $X_n(t)$ has a probability density $u_n(x, t)$ which is the unique solution to system \eqref{evol.intro}.
To this end we will prove uniqueness of weak solutions to \eqref{evol.intro} in the space of measures.

\begin{lemma}\label{lemma:uniqweak}
For every $n\in \mathbb N$ let 
$$D_n:=\left\{G:\overline\Omega\times [0,\infty)\to\bb R\; : G(\cdot, t)\in C\pare{A_n}\cap C^2\pare{B_n},\,  G(x, \cdot)\in C^1\pare{[0, T]},  \frac{\partial G}{\partial \eta}\mid_{\partial B_n}=0\right\}.$$ 
Then,
there exists a unique measure $\nu_t$ solution to 
\begin{align}\label{weak:eq2}
\begin{cases}
\displaystyle \frac{\partial}{\partial t}\int_{\overline\Omega}G(x, t)\nu_t(dx)=\int_{\overline\Omega}\cor{\mathcal{L}_nG(x, t)+\frac{\partial}{\partial t}G(x,t)}\nu_t(dx),&\quad x\in\overline\Omega, \; t\in [0, T],
\\[10pt]
\displaystyle\int_{\overline\Omega}G(x, t)\nu_0(dx)=\int_{\overline\Omega}G(x, t)u_0(x)dx, &\quad x\in\overline\Omega,\; t\in [0, T],
\end{cases}
\end{align}
for every $G\in D_n$.

Such solution is given by $\nu^n_t(dx)=u_n(x, t)dx,$ where
$u_n(x, t)$ is the unique solution to \eqref{evol.intro}.
\end{lemma}

\begin{proof}
The existence of a solution to \eqref{weak:eq2} follows just by taking $\nu_t^n(dx)=u_n(x, t)dx$ where $u_n(x, t)$ is a solution of system \eqref{evol.intro}, see \cite{GQR} for more details. We prove now the uniqueness.
Suppose that there exist two trajectories of measures $\nu_t^n(dx)$ and $\tilde{\nu}_t^n(dx)$ such that \eqref{weak:eq2} holds. Call 
$$
\omega^n_t(dx)=\nu_t^n(dx)-\tilde{\nu}_t^n(dx).
$$ 
The evolution of $\omega^n_t(dx)$ satisfies equation \eqref{weak:eq2} with initial condition $\int_{\overline\Omega}G(x, t)\omega^n_0(dx)=0$.
Therefore, for all $t\in [0, T]$, it holds that
\begin{align}\label{eqg}
\int_{\overline\Omega}G(x, t)\omega^n_t(dx)=\int_0^t\cor{\int_{\overline\Omega} \mathcal{L}_nG(x, s)\omega^n_s(dx)}ds+\int_0^t\cor{\int_{\overline\Omega} \frac{\partial}{\partial s}G(x,s)\omega^n_s(dx)}ds.
\end{align}
For every $f\in C(\overline\Omega)$ and $t\in [0,T]$ consider the function $\widetilde{G_n}(x, s)$ given by the solution to the evolution problem
with time reversed
\begin{align}\label{systtre}
\begin{cases}
\displaystyle \mathcal{L}_n\widetilde{G_n}(x, s)+\frac{\partial}{\partial s}\widetilde{G_n}(x, s)=0,\qquad & x \in \overline\Omega, \, s\in [0, t],\\[10pt]
\displaystyle \widetilde{G_n}(x, t)=f(x),                       & x\in\overline\Omega.
\end{cases}
\end{align}
Observe that $\widetilde{G_n}$ is a solution of \eqref{systtre} if and only if  the function $\widetilde{G_n}(x, t-s)$ satisfies equation \eqref{evol.intro}, in the variables $(x,s)$, with initial condition at time $t=0$ given by $f(x)$.
Therefore, from \cite{GQR}, we know that there exists a $\widetilde{G_n}\in D_n$ solution to \eqref{systtre}.
Replacing $G(x, s)=\widetilde{G_n}(x, s)$ in \eqref{eqg} we get that
\begin{align}
\int f(x)\omega_t^n(dx)=0.
\end{align}
By the arbitrariness of the function $f\in C(\overline\Omega)$ and the time $t$  we can conclude that, for all $t\geq 0$, $\omega_t^n(dx)$ is the null measure and therefore $\nu_t^n(dx)=\tilde \nu_t^n(dx)$.
\end{proof}

As an immediate consequence of this uniqueness result we get that the process $X_n(t)$ has a density.

\begin{corollary}  \label{corol.densidad} The process $X_n(t)$ has a density
that is characterized as the unique solution $u_n(x,t)$ to 
\begin{equation}\label{evol.fffff}
\left\{
\begin{array}{ll}
\displaystyle \frac{\partial u_n}{\partial t} (x,t) = {L}_n u_n(x,t), \qquad & x\in \Omega, \, t>0, \\[10pt]
u_n(x,0)=u_0(x), \qquad & x\in \Omega.
\end{array}
\right.
\end{equation}
\end{corollary}

Consider now the coupled process $\pare{X_n\pare{t}, I_n\pare{t}}\in D\pare{[0, T],\overline\Omega}\times D\pare{[0, T],\brc{1,2}}$. 
As we explained in the introduction, $I_n(t)$ contains the information over the set ($A_n$ or $B_n$) in which $X_n(t)$ is located. More precisely, 
$$I_n\pare{t}=
\left\{\begin{array}{ll}
1 \qquad \mbox{if } X_n(t) \in A_n, \\[10pt] 
2 \qquad \mbox{if } X_n(t) \in B_n.
\end{array} \right.
$$ 
The pair $\pare{X_n\pare{t}, I_n\pare{t}}$ is a Markov process whose generator $\mathcal{L}_n$ is defined on functions  
$$f\in \mathcal T_n:=
\left\{f : \overline\Omega\times \brc{1,2}  \mapsto \mathbb{R},\;: f(\cdot,1) \in C\pare{A_n}, \, f (\cdot,2)\in C^{2}\pare{B_n},  \frac{\partial f}{\partial \eta} (\cdot,2)\mid_{\partial B_n}=0
\right\}$$ as follows,
\begin{align}\label{gen.intro1}
\mathcal{L}_nf(x, i)=\begin{cases}& \displaystyle
\chi_{A_n}\pare{x}\int_\Omega \chi_{A_n}\pare{y}J\pare{x,y}\pare{f\pare{y,1}-f\pare{x,1}}dy \\[10pt]
& \displaystyle \qquad +\chi_{A_n}\pare{x}\int_\Omega \chi_{B_n}\pare{y}J\pare{x,y}\pare{f\pare{y,2}-f\pare{x,1}}dy
\qquad\qquad\qquad\qquad\qquad\text{if $i=1$,} \\[10pt] 
&\displaystyle \chi_{B_n}\pare{x}\int_\Omega \chi_{A_n}\pare{y}J\pare{x,y}\pare{f\pare{y, 1}-f\pare{x, 2}}dy+\frac{1}{2}\chi_{B_n}\pare{x} \Delta f \pare{x, 2}\qquad \qquad\text{if $i=2$.}
\end{cases}
\end{align}
 Now we have a remark on notation, 
we are using $L_n$ for the generator for $X_n(t)$ and $\mathcal{L}_n$ for the generator for $(X_n (t),I_n(t))$.
Observe that a function $f\in \mathcal T_n$ can be thought as a pair of functions
\begin{align}
\begin{cases}
f(x,1) = f_1\pare{x},\\
f(x,2)= f_2\pare{x},
\end{cases}
\end{align}
where $f_1\in  C\pare{A_n}$ and $f_2\in\{f\in C^2\pare{B_n}: \frac{\partial f_1}{\partial \eta}\mid_{\partial B_n}=0\}$.

By Lemma A.1.5.1 of \cite{KL} we know that, for every bounded function $f\in \overline\Omega\times\brc{1,2}\to \bb R$,
\begin{align}\label{f1}
M_n^f\pare{t}=f\pare{X_n\pare{t}, I_n\pare{t}}-f\pare{X_n\pare{0}, I_n\pare{0}}-\int_0^t\mathcal L_nf\pare{X_n\pare{s}, I_n\pare{s}}ds
\end{align}
and
\begin{align}\label{f2}
N_n^f\pare{t}=\pare{M_n^f\pare{t}}^2-\int_0^t\pare{\mathcal L_n\pare{f\pare{X_n\pare{s}, I_n\pare{s}}}^2-2f\pare{X_n\pare{s}, I_n(s)}\mathcal L_nf\pare{X_n\pare{s} I_n(s)}} ds
\end{align}
are martingales with respect to the natural filtration generated by the process.

Let $P_n\in \mathcal M_1\pare{D\pare{[0, T], \overline\Omega}\times D\pare{[0, T], \brc{1,2}}}$ be the law of the process $\pare{X_n\pare{t}, I_n\pare{t}}_{t\in [0, T]}$; in our notation $\mathcal M_1(X)$ denotes the space of probability measures on a metric space $X$. The next lemma guarantees the tightness of the sequence $\pare{P_n}_ {n\in\bb N}$.
\begin{lemma}\label{lemma1}
The sequence of probability measures $\pare{P_n}_{n\in \bb N}$ is tight.
\end{lemma}
\begin{proof}
Let $P_n^1$ and $P_n^2$ the two marginals of $P_n$. Since $D\pare{[0, T], \overline\Omega}\times D\pare{[0, T], \brc{1,2}}$ is endowed with the product topology, in order to conclude, it is enough to show that the marginals $P_n^1$ and $P_n^2$ are tight. 

We start by proving that the sequence $P_n^1$ is tight.
By Theorem 1.3 and Proposition 1.6 of Chapter 4 in \cite{KL}, it is sufficient to show that the following conditions hold:
\begin{enumerate}
\item for every $ t\in [0,T]$ and $\epsilon>0$ there exists a compact set $K(t,\epsilon )\subseteq \overline\Omega$ such that 
$$\sup_{n} P_{n}^1\Big( X\pare{\cdot}: X\pare{t}\not\in K\pare{t,\epsilon} \Big)\leq \epsilon,$$
\item for every $ \epsilon>0$, we have that $$\lim_{\zeta\to 0}\limsup_{n\to +\infty} \sup_{\tau\in \Lambda_T, \theta\leq\zeta}  
P_{n}^1\Big( X\pare{\cdot}: \abs{X\pare{\tau+\theta}-X\pare{\tau}}>\epsilon \Big)=0,$$
where $\Lambda_T$ is the family of all stopping times bounded by $T$.
\end{enumerate}

The first condition is satisfied since $\overline{\Omega}$ is a compact space. To prove the second condition, fix $\tau\in \Lambda_T$, $\epsilon>0$ and observe that, considering the function $g\pare{x, i}=x$ in \eqref{f1}, we get that
$$
M_n^{g}\pare{t}=X_n\pare{t}-X_n\pare{0}-\int_0^t\mathcal L_n g\pare{X_n\pare{s}, I_n\pare{s}}ds.
$$
Therefore,
\begin{align}
\abs{X_n\pare{\tau+\theta}-X_n\pare{\tau}}\leq \abs{\int_\tau^{\tau+\theta}\mathcal L_n  g \pare{X_n\pare{s}, I_n\pare{s}}ds}+\abs{M_n^{g}\pare{\tau+\theta}-M_n^{g}\pare{\tau}}.
\end{align}
Since
\begin{align}
\mathcal L_ng\pare{x,i}=&\chi_{A_n}\pare{x} \chi_{1}\pare{i} \int_\Omega J\pare{x,y}\pare{{y}-{x}}dy
+\chi_{B_n}\pare{x} \chi_{2}\pare{i}\int_\Omega \chi_{A_n}\pare{y}J\pare{x,y}\pare{{y}-{x}}dy,
\end{align}
we have that
\begin{align}\label{c1}
\abs{\int_\tau^{\tau+\theta}\mathcal L_n g \pare{X_n\pare{s}, I_n\pare{s}}ds}\leq C_1\theta,
\end{align}
where $C_1:=2\norm{J}_\infty |\Omega|$. Moreover, by \eqref{f2} we get that
$$
\bb E\pare{\pare{M_n^{g}\pare{\tau+\theta}}^2-\pare{M_n^{g}\pare{\tau}}^2}
=\bb E\pare{\int_{\tau}^{\tau+\theta} \pare{\mathcal L_n\pare{g\pare{X_n\pare{s}, I_n\pare{s}}}^2-2g\pare{X_n\pare{s}, I_n(s)}\mathcal L_ng\pare{X_n\pare{s} I_n(s)}}ds}.
$$
Since
\begin{align}
&\mathcal L_n\pare{g\pare{x, i}}^2-2g\pare{x, i}\mathcal L_ng\pare{x ,i}\\
&\hspace{+25pt}= \chi_{A_n}\pare{x}\chi_{1}\pare{i}\int_\Omega J\pare{x,y}\pare{y^2-x^2}dy+  \chi_{B_n}\pare{x}\chi_{2}\pare{i}\int_\Omega \chi_{A_n}\pare{y}J\pare{x,y}\pare{y^2-x^2}dy \\
&\hspace{+40pt}-2x\cor{\chi_{A_n}\pare{x}\chi_{1}\pare{i}\int_\Omega J\pare{x,y}\pare{y-x}dy+\chi_{B_n}\pare{x}\chi_{2}\pare{i}\int_\Omega \chi_{A_n}\pare{y}J\pare{x,y}\pare{y-x}dy},
\end{align}
we obtain that
\begin{align}
&\bb E\pare{\pare{M_n^{g}\pare{\tau+\theta}}^2-\pare{M_n^{g}\pare{\tau}}^2}\leq C_2\theta,
\end{align}
with $C_2= 8\norm{J}_\infty\abs{\Omega}^3+1$. Therefore, by Markov's inequality 
\begin{equation}\label{c2}
\bb P\pare{\abs{M_n^{g}\pare{\tau+\theta}-M_n^{g}\pare{\tau}}>\epsilon}\leq \frac{\bb E\pare{\pare{M_n^{g}\pare{\tau+\theta}}^2-\pare{M_n^{g}\pare{\tau}}^2}}{\epsilon^2} \leq\frac{C_2\theta}{\epsilon^2},
\end{equation}
for all $\epsilon>0$.
The bounds \eqref{c1} and \eqref{c2} allow to conclude the second condition that guarantees the tightness of the sequence $P_n^1$.

We proceed now in a similar way to prove the tightness of the sequence $P_n^2$. As before it is enough to show that
\begin{enumerate}
\item for every $ t\in [0,T]$ and every $ \epsilon>0$ there exists a compact set $K(t,\epsilon )\subseteq\brc{1,2}$ such that 
$$\sup_{n} P_{n}^2\Big(I\pare{\cdot}: I\pare{t}\not\in K\pare{t,\epsilon}\Big)\leq \epsilon,$$
\item for every $ \epsilon>0$ it holds that $$\lim_{\zeta\to 0}\limsup_{n\to +\infty} \sup_{\tau\in L_T,\theta\leq\zeta}  
P_{n}^2\Big( I\pare{\cdot}: \abs{\Lambda\pare{\tau+\theta}-I\pare{\tau}}>\epsilon \Big) =0.$$
\end{enumerate}
The first condition is trivially satisfied taking $K(t, \epsilon)=\brc{1,2}$. Hence, we need to prove the second condition.
Considering the function $h\pare{x, i}=i$ in \eqref{f1}, we get that
$$
M_n^{h}\pare{t}=I_n\pare{t}-I_n\pare{0}-\int_0^t\mathcal L_n
h\pare{X_n\pare{s}, I_n\pare{s}}ds.
$$
Therefore
\begin{align}
\abs{I_n\pare{\tau+\theta}-I_n\pare{\tau}}\leq \abs{\int_\tau^{\tau+\theta}\mathcal L_nh\pare{X_n\pare{s}, I_n\pare{s}}ds}+
\abs{M_n^{h}\pare{\tau+\theta}-M_n^{h}\pare{\tau}}.
\end{align}
Since
$$
\mathcal L_nh\pare{x,i}=\chi_{A_n}\pare{x} \chi_{1}\pare{i} \int_\Omega \chi_{B_n}\pare{y} J\pare{x,y}dy
-\chi_{B_n}\pare{x} \chi_{2}\pare{i} \int_\Omega \chi_{A_n}\pare{y}J\pare{x,y}dy,
$$
we have that
\begin{align}\label{d1}
\abs{\int_\tau^{\tau+\theta}\mathcal L_n h \pare{X_n\pare{s}, I_n\pare{s}}ds}\leq C_3\theta,
\end{align}
where $C_3=2\norm{J}_\infty |\Omega|$. Moreover, by \eqref{f2}, we get that
\begin{align}
&\bb E\pare{\pare{M_n^{h}\pare{\tau+\theta}}^2-\pare{M_n^{h}\pare{\tau}}^2}\\
&\hspace{+15pt}=\bb E\pare{\int_{\tau}^{\tau+\theta} \pare{\mathcal L_n\pare{h \pare{X_n\pare{s}, I_n\pare{s}}}^2-2 h \pare{X_n\pare{s}, I_n(s)}\mathcal L_n h \pare{X_n\pare{s} I_n(s)}} ds  } \leq C_4\theta,
\end{align}
with $C_4=10 \norm{J}_\infty |\Omega|$. 
The last inequality follows from the fact that
\begin{align}
\mathcal L_n\pare{h\pare{x, i}}^2-2 h \pare{x, i}\mathcal L_nf_2\pare{x, i} = \, & 3\pare{\chi_{A_n}\pare{x}\chi_{1}\pare{i}-\chi_{B_n}\pare{x}\chi_{2}\pare{i}}\int_\Omega \chi_{A_n}\pare{y}J\pare{x,y}dy\\[10pt]
&\ -2i\cor{\chi_{A_n}\pare{x}\chi_{1}\pare{i}\int_\Omega \chi_{B_n}\pare{y}J\pare{x,y}dy-\chi_{B_n}\pare{x}\chi_{2}\pare{i}\int_\Omega \chi_{A_n}\pare{y}J\pare{x,y}dy}.
\end{align}
Finally, by Markov's inequality we get that
\begin{equation}\label{d2}
\bb P\pare{\abs{M_n^{h}\pare{\tau+\theta}-M_n^{h}\pare{\tau}}>\epsilon}\leq \frac{\bb E\pare{\pare{M_n^{h}\pare{\tau+\theta}}^2-\pare{M_n^{h}\pare{\tau}}^2}}{\epsilon^2} \leq\frac{C_4\theta}{\epsilon^2},
\end{equation}
for all $\epsilon>0$.
Bounds \eqref{d1} and \eqref{d2} allow to conclude the second condition that guarantees the tightness of the sequence $P_n^2$.
\end{proof}
Lemma \ref{lemma1} guarantees that the sequence of processes $\pare{X_n(t), I_n(t)}_{t\in [0, T]}$ converges in distribution along subsequences.
In the following theorem we prove that all subsequences converge to the same limit and we characterize the generator of the limit process.

\begin{theorem}\label{Prop:1}
The sequence $\pare{X_n\pare{t}, I_n\pare{t}}$ converges 
$$\pare{X_n\pare{t}, I_n\pare{t}}\xrightarrow[n\to +\infty]{D}\pare{X\pare{t}, I\pare{t}}$$ in $D\pare{[0, T], \overline \Omega}\times D\pare{[0, T], \brc{1,2}}$. 
The limit $\pare{X\pare{t}, I\pare{t}}$ is a Markov process 
 whose generator $\widetilde{\mathcal L}$ is defined on functions $f\in C\pare{\overline\Omega\times\brc{1,2}}$ as
\begin{align}
\widetilde{\mathcal L}f\pare{x, i}= \, &\chi_{1}\pare{i}\brc{\int_{\Omega}\pare{1-\theta\pare{y}}J\pare{x, y}\pare{f\pare{y, 1}-f\pare{x,1}}dy+\int_{\Omega}\theta\pare{y}J\pare{x, y}\pare{f\pare{y,2}-f\pare{x, 1}}dy}\\[10pt]
&+\chi_{2}\pare{i} \brc{\int_{\Omega}\pare{1-\theta\pare{y}} J\pare{x, y}\pare{f\pare{y, 1}-f\pare{x,2}}dy}.
\end{align}
\end{theorem}
\begin{proof}
Lemma \ref{lemma1} implies that any subsequence of $P_n$ has a convergent sub-subsequence; it remains then to characterize all the limit points of the sequence $ P_n$. Let $ \tilde P$ be a limit point and $ P_{n_k}$ be a subsequence converging to $\tilde P$.
To prove the theorem it is enough to show that $\tilde P$ concentrates its mass on trajectories $\pare{X\pare{\cdot}, I\pare{\cdot}}$ such that,
\begin{align}
f\pare{X\pare{t}, I\pare{t}}-f\pare{X\pare{0}, I\pare{0}}-\int_0^t \widetilde{\mathcal L}f\pare{X\pare{s}, I\pare{s}}ds
\end{align}
is a martingale, for every $f\in C\pare{\overline\Omega\times\brc{1,2}, \bb R}$ and for every $t\in [0,T]$. This implies convergence of the entire sequence $P_n$ and caracterizes the limit $\tilde P$ as the law of the Markov process with generator $\widetilde{\mathcal L}$, we refer the reader to Chapter 4 in \cite{EK} for a deeper discussion of the issue.
Therefore, to conclude the proof we need to show that,
\begin{align}\label{pok}
\bb E^{\tilde P}\cor{g\pare{\pare{X\pare{s}, I\pare{s}}, 0\leq s\leq t_0}\pare{f\pare{X\pare{t}, I\pare{t}}-f\pare{X\pare{t_0}, I\pare{t_0}}-\int_{t_0}^t\widetilde{\mathcal L}f\pare{X\pare{s}, I\pare{s}}ds}}=0,
\end{align}
for every bounded continuous function $g:\mathcal D\pare{[0, T], \overline\Omega}\times \mathcal D\pare{[0, T], \brc{1,2}}\to \bb R$,  for every $f\in C\pare{\overline\Omega\times\brc{1,2}}$ and for every $0\leq t_0<t\leq T$. 

By the tightness proved in Lemma \ref{lemma1} we know that
\begin{align}\label{pouy}
\bb E^{\tilde P}&\cor{g\pare{\pare{X\pare{s}, I\pare{s}}, 0\leq s\leq t_0}\pare{f\pare{X\pare{t}, I\pare{t}}-f\pare{X\pare{t_0}, I\pare{t_0}}-\int_{t_0}^t \tilde{\mathcal L}f\pare{X\pare{s}, I\pare{s}}ds}}\\
&=\lim_{n_k\to +\infty}\bb E^{P^{n_k}}\pare{g\pare{\pare{X\pare{s}, I\pare{s}}, 0\leq s\leq t_0}\pare{f\pare{X\pare{t}, I\pare{t}}-f\pare{X\pare{t_0}, I\pare{t_0}}-\int_{t_0}^t \tilde{\mathcal L}f\pare{X\pare{s}, I\pare{s}}ds}}.
\end{align}
Using the approximation lemma, Lemma \ref{lema-phi}, we can take a sequence of functions $f_n\in \mathcal T_n$  such that $\Delta f_n(\cdot, i)=0$ for every $x\in B_n$ and
\begin{align}\label{unifconv}
\sup_{x, i}\abs{f_n\pare{x, i}-f\pare{x,i}}\xrightarrow[n\to +\infty]{}0.
\end{align}
To simplify the notation we define
\begin{align}
F(f, f_n, t)&:=f\pare{X\pare{t}, I\pare{t}}-f_n\pare{X\pare{t}, I\pare{t}}-f\pare{X\pare{t_0}, I\pare{t_0}}+f_n\pare{X\pare{t_0}, I\pare{t_0}}\\[10pt]
&\hspace{+25pt}+\int_{t_0}^t\cor{\widetilde{\mathcal L}f_n\pare{X\pare{s}, I\pare{s}}-\widetilde{\mathcal L}f\pare{X\pare{s}, I\pare{s}}}ds.
\end{align}
By the triangular inequality, we have
\begin{align}\label{ineq1}
&\abs{\bb E^{P^{n_k}}\pare{g\pare{\pare{X\pare{s}, I\pare{s}}, 0\leq s\leq t_0}\pare{f\pare{X\pare{t}, I\pare{t}}-f\pare{X\pare{t_0}, I\pare{t_0}}-\int_{t_0}^t\widetilde{\mathcal L}f\pare{X\pare{s}, I\pare{s}}ds}} }\\[10pt]
&\hspace{+20pt}\leq\abs{\bb E^{P_{n_k}}\pare{g\pare{\pare{X\pare{s}, I\pare{s}}, 0\leq s\leq t_0}\pare{f_{n_k}\pare{X\pare{t}, I\pare{t}}-f_{n_k}\pare{X\pare{t_0}, I\pare{t_0}}-\int_{t_0}^t\mathcal L_{n_k}f_{n_k}\pare{X\pare{s}, I\pare{s}}ds}}}\\[10pt]
&\hspace{+35pt}+ \abs{\bb E^{P^{n_k}}\pare{g\pare{\pare{X\pare{s}, I\pare{s}}, 0\leq s\leq t_0}F\pare{f, f_{n_k}, t}} }\\[10pt]
&\hspace{+35pt}+\abs{\bb E^{P_{n_k}}\pare{g\pare{\pare{X\pare{s}, I\pare{s}}, 0\leq s\leq t_0}\int_{t_0}^t\pare{\mathcal L_{n_k}f_{n_k}\pare{X\pare{s}, I\pare{s}}-\tilde{\mathcal L}f_{n_k}\pare{X\pare{s}, I\pare{s}}}ds}}.
\end{align}

Let us analyze the first term in the right hand side of \eqref{ineq1}. From \eqref{f1}, we have that
\begin{align}
f_{n_k}\pare{X\pare{t}, I\pare{t}}-f_{n_k}\pare{X\pare{t_0}, I\pare{t_0}}-\int_{t_0}^t\mathcal L_{n_k}f_{n_k}\pare{X\pare{s}, I\pare{s}}ds
\end{align}
is a martingale. Therefore, 
$$
\bb E^{P_{n_k}}\pare{g\pare{\pare{X\pare{s}, I\pare{s}}, 0\leq s\leq t_0}\pare{f_{n_k}\pare{X\pare{t}, I\pare{t}}-f_{n_k}\pare{X\pare{t_0}, I\pare{t_0}}-\int_{t_0}^t\mathcal L_{n_k}f_{n_k}\pare{X\pare{s}, I\pare{s}}ds}}
=0.
$$

Concerning the second term, we have
\begin{align}
\abs{\bb E^{P^{n_k}}\pare{g\pare{\pare{X\pare{s}, I\pare{s}}, 0\leq s\leq t_0}F\pare{f, f_{n_k}, t}}}\leq 6\norm{g}_\infty\norm{f-f_{n_k}}_\infty,
\end{align}
which, by \eqref{unifconv}, vanishes as $n_k\to 0$. Therefore, to conclude \eqref{pok} we just need to show that
\begin{align}
\lim_{n_k\to \infty}\abs{\bb E^{P_{n_k}}\pare{g\pare{\pare{X\pare{s}, I\pare{s}}, 0\leq s\leq t_0}\int_{t_0}^t\pare{\mathcal L_{n_k}f_{n_k}\pare{X\pare{s}, I\pare{s}}-\widetilde{\mathcal L}f_{n_k}\pare{X\pare{s}, I\pare{s}}}ds}}=0.
\end{align}
Since
\begin{align}
&\abs{\bb E^{P_{n_k}}\pare{g\pare{\pare{X\pare{s}, I\pare{s}}, 0\leq s\leq t_0}\int_{t_0}^t\pare{\mathcal L_{n_k}f_{n_k}\pare{X\pare{s}, I\pare{s}}-\widetilde{\mathcal L}f_{n_k}\pare{X\pare{s}, I\pare{s}}}ds}}\\[10pt]
&\hspace{+15pt}\leq\norm{g}_\infty\bb E^{P_{n_k}}\pare{\int_{t_0}^t\abs{\mathcal L_{n_k}f_{n_k}\pare{X\pare{s}, I\pare{s}}-\widetilde{\mathcal L}f_{n_k}\pare{X\pare{s}, I\pare{s}}}ds}
\end{align}
it is enough to prove that
\begin{align}\label{tesi3}
\lim_{n_k\to \infty}\bb E^{P_{n_k}}
\pare{\int_{t_0}^t\abs{\mathcal L_{n_k}f_{n_k}\pare{X\pare{s}, I\pare{s}}-\widetilde{\mathcal L}f_{n_k}\pare{X\pare{s}, I\pare{s}}}ds}=0.
\end{align}
Denoting by $\mathcal D:= \mathcal D\pare{[0, T], \overline\Omega}\times \mathcal D\pare{[0, T], \brc{1,2}}$ and using Fubini's theorem  we get
\begin{align}\label{equal1}
&\bb E^{P_{n_k}}\pare{\int_{t_0}^t\abs{\mathcal L_{n_k}f_{n_k}\pare{X\pare{s}, I\pare{s}}-\widetilde{\mathcal L}f_{n_k}\pare{X\pare{s}, I\pare{s}}}ds}\\
&\hspace{+15pt}=\int_{t_0}^t\int_{\mathcal D}\abs{\mathcal L_{n_k}f_{n_k}\pare{X\pare{s}, I\pare{s}}-\widetilde{\mathcal L}f_{n_k}\pare{X\pare{s}, I\pare{s}}}dP_{n_k}\pare{X, I}ds.
\end{align}
Observe that $I_n(s)=1+\chi_{{ B_n}}\pare{X_n(s)}$ and $\chi_{A_n}\pare{X_n(s)}=\chi_{1}\pare{I_n(s) }$. Then, recalling that  $u_n(x,s)$ is the probability density of the process $X_n(s)$, we get
\begin{align}
&\int_{t_0}^t\int_{\mathcal D}\abs{\mathcal L_{n_k}f_{n_k}\pare{X\pare{s}, I\pare{s}}-\widetilde{\mathcal L}f_{n_k}\pare{X\pare{s}, I\pare{s}}}dP_{n_k}\pare{X, I}ds\\
&\hspace{+15pt}=\int_{t_0}^t\int_{\Omega}\abs{\mathcal L_{n_k}f_{n_k}\pare{x, 1+\chi_{{B_{n_k}}}\pare{x}}-\widetilde{\mathcal L}f_{n_k}\pare{x, 1+\chi_{{B_{n_k}}}\pare{x}}}u_{n_k}\pare{x,s}dxds\\[10pt]
&\hspace{+15pt}\leq\int_{t_0}^t\int_\Omega\abs{\Upsilon_{n_k}^1\pare{x}+\Upsilon_{n_k}^2\pare{x}+\Upsilon_{n_k}^3\pare{x}+\Upsilon_{n_k}^4\pare{x}}u_{n_k}\pare{x, s}dxds,
\end{align}
where
\begin{align}
&\Upsilon_{n_k}^1\pare{x}=\chi_{A_{n_k}}\pare{x}\int_\Omega \chi_{A_n}\pare{y}J\pare{x,y}\pare{f_{n_k}\pare{y, 1}-f_{n_k}\pare{x, 1}}dy\\[10pt]
&\hspace{+100pt}- \chi_{A_{n_k}}\pare{x}\int_\Omega \pare{1-\theta\pare{y}}J\pare{x, y}\pare{f_{n_k}\pare{y,1}-f_{n_k}\pare{x,1}}dy,\\[10pt]
&\Upsilon_{n_k}^2\pare{x}=\chi_{A_{n_k}}\pare{x}\int_\Omega \chi_{B_{n_k}}\pare{y}J\pare{x,y}\pare{f_{n_k}\pare{y, 2}-f_{n_k}\pare{x, 1}}dy\\[10pt]
&\hspace{+100pt}- \chi_{A_{n_k}}\pare{x}\int_\Omega \theta\pare{y}J\pare{x, y}\pare{f_{n_k}\pare{y, 2}-f_{n_k}\pare{x, 1}}dy,\\[10pt]
&\Upsilon_{n_k}^3\pare{x}=\chi_{B_{n_k}}\pare{x}\int_\Omega \chi_{A_{n_k}}\pare{y}J\pare{x,y}\pare{f_{n_k}\pare{y,1}-f_{n_k}\pare{x,2}}dy\\[10pt]
&\hspace{+100pt}-\chi_{B_{n_k}}\pare{x} \int_\Omega \pare{1-\theta\pare{y}}J\pare{x, y}\pare{f_{n_k}\pare{y,1}-f_{n_k}\pare{x, 2}}dy,\\[10pt]
&\Upsilon_{n_k}^4\pare{x}=\frac{1}{2}\chi_{B_{n_k}}\pare{x}\Delta f_{n_k}\pare{x,2}.
\end{align}
Therefore, \eqref{tesi3} is proved once we show that
\begin{align}\label{fourp}
\lim_{n_k\to \infty}\int_{t_0}^t\int_{\Omega}\abs{\Upsilon_{n_k}^i\pare{x}}u_{n_k}(x, s)ds=0,\qquad\forall i\in \brc{1,2,3, 4}.
\end{align}
Since $u_n(x, s)\chi_{A_n}(x)=a_n(x,s)$, we get
\begin{align}\label{form1}
&\int_{t_0}^t\int_\Omega\abs{\Upsilon^1_{n_k}(x)}u_{n_k}\pare{x,s}dxds\\[10pt]
&\hspace{+0pt}=\int_{t_0}^t\int_{\Omega}\abs{\int_\Omega \chi_{A_{n_k}}\pare{y}J\pare{x,y}\pare{f_{n_k}\pare{y,1}-f_{n_k}\pare{x,1}}dy- \int_\Omega \pare{1-\theta\pare{y}}J\pare{x, y}\pare{f_{n_k}\pare{y,1}-f_{n_k}\pare{x,1}}dy}a_{n_k}\pare{x, s}dxds\\[10pt]
&\hspace{+0pt}\leq\int_{t_0}^t\int_{\Omega}\abs{\int_\Omega \pare{\chi_{A_{n_k}}\pare{y}-\pare{1-\theta(y)}}J\pare{x,y}f_{n_k}\pare{y,1}dy}a_{n_k}\pare{x, s}dxds\\[10pt]
&\hspace{+15pt}+\int_{t_0}^t\int_{\Omega}\abs{\int_\Omega \pare{\chi_{A_{n_k}}\pare{y}-\pare{1-\theta(y)}}J\pare{x,y}dy}\abs{f_{n_k}\pare{x, 1}}a_{n_k}\pare{x,s}dxds.
\end{align}
By the continuity of $J$, limit \eqref{unifconv} and the fact that $\chi_{A_n} (x) \rightharpoonup 1-\theta (x)$ (see \eqref{cond.sets}) we get
\begin{align}\label{mj1}
\sup_{x\in \Omega}\abs{\int_\Omega \Big\{\chi_{A_n}\pare{y}-\pare{1-\theta(y)}\Big\}f_{n_k}(y, 1)J\pare{x,y}dy }\xrightarrow[n\to +\infty]{}0,
\end{align}
and
\begin{align}\label{mj2}
\sup_{x\in \Omega}\abs{\int_\Omega \Big\{\chi_{A_n}\pare{y}-\pare{1-\theta(y)}\Big\}J\pare{x,y}dy}\xrightarrow[n\to +\infty]{}0.
\end{align}
Recall that $a_{n_k}(x,s)\rightharpoonup a(x,s)$ (see Theorem \ref{th32}) and that \eqref{unifconv} holds. By \eqref{mj1} and \eqref{mj2}, we obtain that
the right hand side of \eqref{form1} converges to $0$ as $n\to +\infty$.

Analogously, we have that
\begin{align}\label{form2}
&\int_{t_0}^t\int_\Omega\abs{\Upsilon^2_{n_k}(x)}u_{n_k}\pare{x,s}dxds\\[10pt]
&\hspace{+0pt}=\int_{t_0}^t\int_{\Omega}\abs{\int_\Omega \chi_{B_{n_k}}\pare{y}J\pare{x,y}\pare{f_{n_k}\pare{y, 2}-f_{n_k}\pare{x, 1}}dy- \int_\Omega \theta\pare{y}J\pare{x, y}\pare{f_{n_k}\pare{y, 2}-f_{n_k}\pare{x, 1}}dy}a_{n_k}\pare{x, s}dxds\\[10pt]
&\hspace{+0pt}\leq\int_{t_0}^t\int_{\Omega}\abs{\int_\Omega \pare{\chi_{B_{n_k}}\pare{y}-\theta(y)}J\pare{x,y}f_{n_k}\pare{y, 2}dy}a_{n_k}\pare{x, s}dxds
\\[10pt]
&\hspace{+15pt}+\int_{t_0}^t\int_{\Omega}\abs{\int_\Omega \pare{\chi_{A_{n_k}}\pare{y}-\theta(y)}J\pare{x,y}dy}\abs{f_{n_k}\pare{x, 1}}a_{n_k}\pare{x,s}dxds.
\end{align}
Arguing as before we can conclude that the right hand side of \eqref{form2} goes to zero.

The same ideas apply to 
\begin{align}\label{form3}
&\int_{t_0}^t\int_{\overline\Omega}\abs{\Upsilon^3_{n_k}(x)}u_{n_k}\pare{x,s}dxds\\[10pt]
&\hspace{+0pt}=\int_{t_0}^t\int_{\overline\Omega}\abs{\int_{\overline\Omega} \chi_{A_{n_k}}\pare{y}J\pare{x,y}\pare{f_{n_k}\pare{y, 1}-f_{n_k}\pare{x, 2}}dy- \int_{\overline\Omega} \pare{1-\theta\pare{y}}J\pare{x, y}\pare{f_{n_k}\pare{y, 1}-f_{n_k}\pare{x, 2}}dy}b_{n_k}\pare{x, s}dxds\\[10pt]
&\hspace{+0pt}\leq\int_{t_0}^t\int_{\overline\Omega}\abs{\int_{\overline\Omega} \pare{\chi_{A_{n_k}}\pare{y}-\pare{1-\theta(y)}}J\pare{x,y}f_{n_k}\pare{y, 1}dy}b_{n_k}\pare{x, s}dxds\\[10pt]
&\hspace{+15pt}+\int_{t_0}^t\int_{\overline\Omega}\abs{\int_{\overline\Omega} \pare{\chi_{A_{n_k}}\pare{y}-\pare{1-\theta(y)}}J\pare{x,y}dy}\abs{f_{n_k}\pare{x, 2}}b_{n_k}\pare{x,s}dxds.
\end{align}
Again we can conclude that the right hand side of \eqref{form3} converges to $0$ as $n_k\to \infty$.

Finally, observe that $\Upsilon^4_{n_k}\equiv 0$ because, by construction, $f_{n_k}$ is constant on $B_{n_k}^j$. 

This concludes the proof of \eqref{fourp}.
\end{proof}

We  can prove now the last statement of Theorem \ref{teo.main.intro}, i.e., that the distribution of the limit process $\pare{X(t), I(t)}$ is characterized by the densities $a(x, t)$ and $b(x, t)$.

First of all, observe that, for every measurable $E\subseteq \overline\Omega$,
\begin{align}
\bb P\pare{\pare{X_n(0), I_n(0)}\in E\times\brc{1}}&=\bb P\pare{X_n(0)\in E\cap A_n}\\
&=\int_{E\cap A_n}u_0(x)dx\\
&=\int_E u_0(x)\chi_{A_n}(x)dx\xrightarrow[n\to\infty]{}\int_E u_0(x)\pare{1-\theta\pare{x}}dx.
\end{align}
Therefore, by the tightness result proved in Lemma \ref{lemma1}, we can write
\begin{align}\label{ic1}
\bb P\pare{\pare{X(0), I(0)}\in E\times\brc{1}}=\int_E u_0(x)\pare{1-\theta\pare{x}}dx.
\end{align}
Analogously, we get that
\begin{align}\label{ic2}
\bb P\pare{\pare{X(0), I(0)}\in E\times\brc{2}}=\int_E u_0(x)\theta\pare{x}dx.
\end{align}

Let $\mu(dx, i)=\pare{\mu_t(dx, i)}_{t\in [0, T]}$ be the law of the limit process $\pare{X(t), I(t)}_{t\in [0, T]}$. We can decompose
\begin{align}\label{decmea}
\mu_t(dx, i)=\bb \chi_{1}(i)\mu_t(dx, 1)+\bb \chi_{2}(i)\mu_t(dx, 2)
\end{align}
where, by \eqref{ic1} and \eqref{ic2}, $\pare{\mu_t(dx, 1)}_{t\in [0, T]}$ and $\pare{\mu_t(dx, 2)}_{t\in [0, T]}$ are such that
\begin{align}\label{ic}
\mu_0(dx, 1)=u_0(x)\pare{1-\theta\pare{x}}dx \quad \text{ and }\quad \mu_0(dx, 2)=u_0(x)\theta(x)dx.
\end{align}
Since $\widetilde{\mathcal L}$ is the generator of the process $\pare{X(t), I(t)}$ (see Theorem \ref{Prop:1}), by Lemma A.5.1 of \cite{KL} we can conclude that 
\begin{align}\label{endp}
&\displaystyle \frac{\partial}{\partial t}\sum_{i=1}^2\int_{\overline\Omega}f(x, i)\mu_t(dx, i)=\sum_{i=1}^2\int_{\overline\Omega}\widetilde{\mathcal L}f(x, i)\mu_t(dx, i),
\end{align}
for all bounded $f:\overline\Omega\times \brc{1,2}\to \bb R$.
Therefore, fixing $g\in C\pare{\overline\Omega}$ and choosing $f(x, i)=g(x)\chi_{1}(i)$, we get
\begin{align}\label{eqm1}
\frac{\partial }{\partial t}\int_{\overline\Omega}g(x)\mu_t(dx, 1)&=\int_{\overline\Omega}\int_{\overline\Omega}\pare{1-\theta\pare{y}}J\pare{x, y}\pare{g\pare{y}-g\pare{x}}dy\;\mu_t(dx,1)\\[10pt]
&\qquad -\int_{\overline\Omega}\int_{\overline\Omega}\theta\pare{y}J\pare{x, y}g\pare{x}dy\;\mu_t\pare{dx, 1}\\[10pt]
&\qquad +\int_{\overline\Omega}\int_{\overline\Omega}\pare{1-\theta\pare{y}} J\pare{x, y}g\pare{y}dy\;\mu_t(dx, 2).
\end{align}
Choosing  $f(x, i)=g(x)\chi_{2}(i)$ we get
\begin{align}\label{eqm2}
\frac{\partial }{\partial t} \int_{\overline\Omega}g(x)\mu_t(dx, 2)&=\int_{\overline\Omega}\int_{\overline\Omega}\theta\pare{y}J\pare{x, y}g\pare{y}dy\;\mu_t\pare{dx, 1}\\[10pt]
&\qquad -\int_{\overline\Omega}\int_{\overline\Omega}\pare{1-\theta\pare{y}} J\pare{x, y}g\pare{x}dy\;\mu_t(dx, 2).
\end{align}
We analyse now the right hand side of \eqref{eqm1}. Since $J$ is symmetric, by a change of variables, we can write
\begin{align}\label{adj1}
\int_{\overline\Omega}\int_{\overline\Omega}\pare{1-\theta\pare{y}}J\pare{x, y}\pare{g\pare{y}-g\pare{x}}dy\;\mu_t\pare{dx, 1} &=\int_{\overline\Omega}\int_{\overline\Omega}\pare{1-\theta(x)}J(x,y)g(x)dx\;\mu_t\pare{dy, 1}\\
 &\hspace{+15pt}-\int_{\overline\Omega}\int_{\overline\Omega}\pare{1-\theta(y)}J(x,y)g(x)dy\;\mu_t\pare{dx, 1}.
\end{align}
Moreover, it holds that
\begin{align}\label{adj3}
&\int_{\overline\Omega}\int_{\overline\Omega}\pare{1-\theta\pare{y}} J\pare{x, y}g\pare{y}dy\;\mu_t(dx, 2)=\int_{\overline\Omega}\int_{\overline\Omega}\pare{1-\theta\pare{x}} J\pare{x, y}g\pare{x}dx\;\mu_t(dy, 2).
\end{align}
Replacing \eqref{adj1} and \eqref{adj3} in the right hand side of \eqref{eqm1} we obtain
\begin{align}\label{wqer}
\frac{\partial }{\partial t} \int_{\overline\Omega}g(x)\mu_t(dx, 1)&=\int_{\overline\Omega}\int_{\overline\Omega}g\pare{x}J(x,y)\pare{1-\theta(x)}dx\;\mu_t\pare{dy, 1}\\[10pt]
 &\hspace{+15pt}-\int_{\overline\Omega}\int_{\overline\Omega}g(x)J(x,y)dy\;\mu_t\pare{dx, 1}\\[10pt]
&\hspace{+15pt}+\int_{\overline\Omega}\int_{\overline\Omega}g\pare{x}J\pare{x, y}\pare{1-\theta\pare{x}} dx\;\mu_t(dy, 2).
\end{align}
As before, via a change of variable in the right hand side of \eqref{eqm2}, we can write
\begin{align}\label{adj2}
\frac{\partial }{\partial t} \int_{\overline\Omega}g(x)\mu_t(dx, 2)&=\int_{\overline\Omega}\int_{\overline\Omega}\theta\pare{x}J\pare{x, y}g\pare{x}dx\;\mu_t\pare{dy, 1}\\[10pt]
&\hspace{+15pt}-\int_{\overline\Omega}\int_{\overline\Omega}\pare{1-\theta\pare{y}} J\pare{x, y}g\pare{x}dy\;\mu_t(dx, 2).
\end{align}
Moreover, by \eqref{ic} we get that
\begin{align}\label{ic3}
\int_{\overline\Omega}g(x)\mu_0(dx, 1)=\int_{\overline\Omega}g(x)u_0(x)\pare{1-\theta\pare{x}}dx\quad \text{and} \quad \int_{\overline\Omega}g(x)\mu_0(dx, 2)=\int_{\overline\Omega}g(x)u_0(x)\theta\pare{x}dx.
\end{align}
By Lemma \ref{lemma:weaksys} below we know that there exists a unique pair of trajectories of measures $\pare{\mu_t\pare{dx, 1}, \mu_t\pare{dx, 2}}_{t\in [0,T]}$ which, for every $g\in C\pare{\overline\Omega}$, satisfies \eqref{wqer}, \eqref{adj2} and \eqref{ic2}.
Such pair is given by $\mu_t\pare{dx, 1}=a(x, t)dx$ and 
$\mu_t\pare{dx, 2}=b(x, t)dx$, where the couple $\pare{a(x, t), b(x, t)}$ is the unique solution to system \eqref{sys1.intro}. This concludes the proof of Theorem \ref{teo.main.intro}.

\begin{lemma}\label{lemma:weaksys}
There exists a unique pair $\pare{\mu_t\pare{dx, 1}, \mu_t\pare{dx, 2}}$ such that, for every $g\in C\pare{\overline\Omega}$, it holds that
\begin{align}\label{sys12.11}
\begin{cases}
\displaystyle \frac{\partial }{\partial t} \int_{\overline\Omega}g(x)\mu_t(dx, 1)& \displaystyle =\int_{\overline\Omega}\int_{\overline\Omega}g\pare{x}J(x,y)\pare{1-\theta(x)}dx\;\mu_t\pare{dy, 1}\\[10pt]
 & \displaystyle\hspace{+15pt}-\int_{\overline\Omega}\int_{\overline\Omega}g(x)J(x,y)dy\;\mu_t\pare{dx, 1}\\[10pt]
& \displaystyle \displaystyle\hspace{+15pt}+\int_{\overline\Omega}\int_{\overline\Omega}g\pare{x}J\pare{x, y}\pare{1-\theta\pare{x}} dx\;\mu_t(dy, 2),
\end{cases}
\end{align}
\begin{align}\label{sys12.22}
\begin{cases}
\displaystyle \frac{\partial }{\partial t} \int_{\overline\Omega}g(x)\mu_t(dx, 2)& \displaystyle=\int_{\overline\Omega}\int_{\overline\Omega}\theta\pare{x}J\pare{x, y}g\pare{x}dx\;\mu_t\pare{dy, 1}\\[10pt]
&\displaystyle\hspace{+15pt}-\int_{\overline\Omega}\int_{\overline\Omega}\pare{1-\theta\pare{y}} J\pare{x, y}g\pare{x}dy\;\mu_t(dx, 2),
\end{cases}
\end{align}
with 
\begin{align}\label{sys12.33}
&\displaystyle \int_{\overline\Omega}g(x)\mu_0(dx, 1)=\int_{\overline\Omega}g(x)u_0(x)\pare{1-\theta\pare{x}}dx \quad \mbox{ and } \quad 
\int_{\overline\Omega}g(x)\mu_0(dx, 2)=\int_{\overline\Omega}g(x)u_0(x)\theta\pare{x}dx.
\end{align}

Such solution is given by $\pare{\mu_t\pare{dx, 1}, \mu_t\pare{dx, 2}}=\pare{a(x, t)dx, b(x, t)dx}$, where the pair $\pare{a(x, t), b(x, t)}$ is the unique solution to system \eqref{sys1.intro}.
\end{lemma}

\begin{proof}
The fact that the pair $\pare{a(x, t)dx, b(x, t)dx}$ is a solution of system \eqref{sys12.11}--\eqref{sys12.22}--\eqref{sys12.33} is a consequence of Theorem \ref{th32}. We prove now the uniqueness. Suppose that there exist two pairs $\pare{\nu_t\pare{dx, 1}, \nu_t\pare{dx, 2}}$ and $\pare{\tilde\nu_t\pare{dx, 1}, \tilde\nu_t\pare{dx, 2}}$ for which system \eqref{sys12.11}--\eqref{sys12.22}--\eqref{sys12.33} is satisfied. Let 
$$
\omega_t(dx, 1):=\nu_t\pare{dx, 1}-\tilde\nu_t\pare{dx, 1}\qquad \mbox{ and }\qquad \omega_t(dx, 2):=\nu_t\pare{dx, 2}-\tilde\nu_t\pare{dx, 2}.
$$
Therefore, we know that, for all $g\in C(\overline\Omega)$, 
\begin{align}\label{sys13}
\displaystyle \int_{\overline\Omega}g(x)\omega_t(dx, 1)&=\int_0^t\int_{\overline\Omega}\int_{\overline\Omega}g\pare{x}J(x,y)\pare{1-\theta(x)}dx\;\omega_s\pare{dy, 1}ds\\[10pt]
 &\hspace{+15pt}-\int_0^t\int_{\overline\Omega}\int_{\overline\Omega}g(x)J(x,y)dy\;\omega_s\pare{dx, 1}ds\\[10pt]
&\displaystyle\hspace{+15pt}+\int_0^t\int_{\overline\Omega}\int_{\overline\Omega}g\pare{x}J\pare{x, y}\pare{1-\theta\pare{x}} dx\;\omega_t(dy, 2)ds
\end{align}
and
\begin{align}\label{sys14}
\int_{\overline\Omega}g(x)\omega_t(dx, 2)&=\int_0^t\int_{\overline\Omega}\int_{\overline\Omega}\theta\pare{x}J\pare{x, y}g\pare{x}dx\;\omega_s\pare{dy, 1}ds\\[10pt]
&\displaystyle\hspace{+15pt}-\int_0^t\int_{\overline\Omega}\int_{\overline\Omega}\pare{1-\theta\pare{y}} J\pare{x, y}g\pare{x}dy\;\omega_s(dx, 2)ds,
\end{align}
with initial conditions
\begin{align}\label{sys15}
\int_{\overline\Omega}g(x)\omega_0(dx, 1)=\int_{\overline\Omega}g(x)\omega_0(dx, 2)=0.
\end{align}

In what follows, for all measures $\mu$ on $\overline\Omega$, we denote by $$\displaystyle{\norm{\mu}_{\text{TV}}:=\sup_{g\in C\pare{\overline\Omega}: \norm{g}_\infty\leq 1}\int_{\overline\Omega}g(x)\mu(dx)},$$ the dual norm (total variation) of $\mu$. From \eqref{sys13} and \eqref{sys15} we get
\begin{align}\label{gron1}
\displaystyle \norm{\omega_t(dx, 1)}_{\text{TV}}=\sup_{g\in C\pare{\overline\Omega}: \norm{g}_\infty\leq 1}\bigg\{&\int_0^t\int_{\overline\Omega}\int_{\overline\Omega}g\pare{x}J(x,y)\pare{1-\theta(x)}dx\;\omega_s\pare{dy, 1}ds\\[10pt]
 &\hspace{+15pt}-\int_0^t\int_{\overline\Omega}\int_{\overline\Omega}g(x)J(x,y)dy\;\omega_s\pare{dx, 1}ds\\[10pt]
&\displaystyle\hspace{+15pt}+\int_0^t\int_{\overline\Omega}\int_{\overline\Omega}g\pare{x}J\pare{x, y}\pare{1-\theta\pare{x}} dx\;\omega_t(dy, 2)ds\bigg\}\\[10pt]
&\hspace{-80pt}\leq C\int_0^t\pare{\norm{\omega_s(dx, 1)}_{\text{TV}}+\norm{\omega_s(dx, 2)}_{\text{TV}}}ds ,
\end{align}
where $C=2\norm{J}_\infty\abs{\Omega}^2$. Analogously, by \eqref{sys14} and \eqref{sys15}, we obtain
\begin{align}\label{gron2}
\displaystyle \norm{\omega_t(dx, 2)}_{\text{TV}}\leq C\int_0^t\pare{\norm{\omega_s(dx, 1)}_{\text{TV}}+\norm{\omega_s(dx, 2)}_{\text{TV}}}ds.
\end{align}
Now, from \eqref{gron1} and \eqref{gron2}, using Gronwall's inequality, we conclude that
\begin{align}
\displaystyle \norm{\omega_t(dx, 1)}_{\text{TV}}+\norm{\omega_t(dx, 2)}_{\text{TV}}=0.
\end{align}
Therefore, $\omega_t(dx, 1)$ and $\omega_t(dx, 2)$ coincide with the null measure on $\overline \Omega$ and consequently $\nu_t(dx, i)=\tilde \nu_t(dx, i)$, for $i\in \brc{1,2}$. This concludes the proof.
\end{proof}

\begin{remark} {\rm It holds that
$$\pare{X_n(t)}_{t\in [0, T]}\xrightarrow[n\to \infty]{D}\pare{X(t)}_{t\in [0, T]},$$ where $X(t)$ has probability density $$u(x, t)=a(x, t)+b(x, t).$$ 
Indeed, the convergence in distribution to the process $\pare{X(t)}_{t\in [0, T]}$ is a consequence of \eqref{convd}. 
Moreover, since $X(t)$ is the marginal in the first variable of $\pare{X(t), I(t)}$, we can write
\begin{align}
\displaystyle \bb P\pare{X(t)\in E}&=\sum_{i=1}^2\bb P\pare{X(t)\in E, I(t)=i}=\sum_{i=1}^2\int_E\mu_t\pare{dx, i}=\int_E\pare{a(x, t)+b(x, t)}dx
=\int_E u(x, t) dx,
\end{align}
for every measurable set $E\subseteq \overline\Omega$.}
\end{remark}

\section{Examples} \label{sect-generaliz}

In this section we collect some simple examples that illustrate our main result.

\subsection{The case $\theta$ constant.} \label{sect-1-d}

First, we deal with the simplest case in which the limit $\theta$ is just a constant $\theta =k$. In this case 
the limit system reads as,
\begin{equation}\label{sys1.k.11}
\left\{
\begin{array}{ll}
\displaystyle \frac{\partial a}{\partial t}\pare{x, t}=\int_{\Omega}J\pare{x, y}\pare{a\pare{y,t}-a\pare{x,t}}dy 
-k\int_{\Omega}J\pare{x, y}a\pare{y,t}dy+(1-k) \int_{\Omega}J\pare{x,y}b\pare{y,t}dy \quad & x\in \Omega , \, t>0,\\[10pt]
\displaystyle \frac{\partial b}{\partial t}\pare{x, t}= k \int_{\Omega} J\pare{x, y}a\pare{y,t}dy- (1-k) b\pare{x,t}\int_{\Omega} 
J\pare{x, y} dy
\quad & x\in \Omega , \, t>0,\\[10pt]
a\pare{x, 0}=\pare{1- k }u_0\pare{x}, \quad b\pare{x, 0}=k u_0\pare{x}
\qquad & x\in \Omega.
\end{array} \right.
\end{equation}

In the special case $k=1/2$ we get
\begin{equation}\label{sys1.k.}
\left\{
\begin{array}{ll}
\displaystyle \frac{\partial a}{\partial t}\pare{x, t}= \frac12 \int_{\Omega}J\pare{x, y}\pare{a\pare{y,t}-a\pare{x,t}}dy 
-\frac12 a(x,t) \int_{\Omega}J\pare{x, y} dy+ \frac12 \int_{\Omega}J\pare{x,y}b\pare{y,t}dy  & x\in \Omega , \, t>0,\\[10pt]
\displaystyle \frac{\partial b}{\partial t}\pare{x, t}= \frac12 \int_{\Omega} J\pare{x, y}a\pare{y,t}dy- \frac12 b\pare{x,t}\int_{\Omega} 
J\pare{x, y} dy
\quad & x\in \Omega , \, t>0,\\[10pt]
\displaystyle a\pare{x, 0}= \frac12 u_0\pare{x}, \quad b\pare{x, 0}= \frac12  u_0\pare{x}
\qquad & x\in \Omega.
\end{array} \right.
\end{equation}

{\bf The $1-$d case.}

As a simple case in which we have $\theta =k $ is in the one-dimensional case when $A_n$, $B_n$ are constructed
alternating small intervals.

Fix $n\in \bb N$ and divide the interval $[0,1]$ into $n$ subintervals that we call $\{I_n^j\}_{j\in \brc{1, \ldots, n}}$, each one of the same 
length, $\frac{1}{n}$. For each $j\in\{1, \ldots, n\}$ the subinterval $I_n^j$ is written as $I_n^j=A_n^j\cup B_n^j$ where $A_n^j$ and $B_n^j$ are disjoint sets such that 
$|{A_n^j}|=\frac{1-k}{n}$ and $|{B_n^j}|=\frac{k}{n}$, with $k\in (0,1)$ a fixed number (we use the same proportion 
between $A_n$ and $B_n$ inside every interval).

To obtain $k=1/2$ we just divide $[0,1]$ into
small intervals of the same length and collect the even ones as $A_n$ and the odd ones as $B_n$.

\begin{remark} {\rm
Just as a curiosity, we observe in this configuration, $[0,1] = A_n \cup B_n$, as above, 
we can obtain that the functions in the approximating sequence given in Lemma \ref{lema-phi}
 can be taken to be continuous (as we already observed in Remark \ref{rem.cont}). In this simple 1-d case the
 construction can be made explicit. We drop the dependence on $t$ to simplify the notation.
Consider a non decreasing function $h: [0,1] \mapsto [0,1]$, $h\in C^\infty\pare{[0,1]}$ for which there exists $\mathcal E_0$ and $\mathcal E_1$, neighbourhoods of $0$ and $1$ respectively, such that $h_{|\mathcal E_0}\equiv 0$ and $h_{|\mathcal E_1}\equiv 1$. 
Fix $\phi \in C\pare{[0,1], \mathbb R}$. We can define the approximating functions $\phi_n$ as follows
\begin{align}
\phi_n\pare{x}=
\begin{cases}
\displaystyle h\pare{\frac{x-\frac{j-u}{n}}{\frac{u}{n}}}\kint_{B^n_{j+1}}\phi(y)dy+\pare{1-h\pare{\frac{x-\frac{j-u}{n}}{\frac{u}{n}}}}\vint_{B^n_{j}}\phi(y)dy\quad\text{if }x\in A_n^j,\\[5pt]
\displaystyle \kint_{B_n^j}\phi(x)dx\hspace{+245pt}\text{if }x\in B^n_j.
\end{cases}
\end{align}
By definition the sequence of functions $\pare{\phi_n}_{n\in\bb N}$ is such that $$\sup_{n\in \bb N}\norm{\phi_n}_\infty\leq \norm{\phi}_\infty\pare{1+2\norm{h}_\infty}.$$ Let us show that 
\begin{align}\label{ucv33}
\norm{\phi-\phi_n}_\infty\xrightarrow[n\to +\infty]{}0.
\end{align}
As before, we have that since the function $\phi$ is uniformly continuous in the interval $[0,1]$ then for every $\epsilon>0$ there exists $\delta_\epsilon$ such that 
\begin{align}\label{acf22}
\abs{\phi(x)-\phi(y)}<\epsilon, \qquad \forall x,y :\abs{x-y}<\delta_\epsilon.
\end{align}
Fix $\epsilon>0$ and take $x\in \mathcal B_n$, consequently there exists $j\in\brc{1, \ldots, N}$ such that $x\in B_n^j$. Therefore, as we did before, we have
\begin{align}\label{pl1}
\abs{\phi_n\pare{x}-\phi\pare{x}} \displaystyle = \abs{\kint_{B_n^j}\phi(z)dz-\phi(x)} \leq \epsilon,
\end{align}
for every $n\geq \frac{1}{\delta_\epsilon}$. 
Instead if $x\in \mathcal A_n$ there exists $j\in\brc{1, \ldots, N}$ such that $x\in A^j_n$. Therefore,
\begin{align}\label{pl2}
\abs{\phi_n\pare{x}-\phi\pare{x}}&\leq\abs{\phi_n\pare{x}-\phi_n\pare{\frac{j}{n}}}+\abs{\phi_n\pare{\frac{j}{n}}-\phi\pare{x}}\\[10pt]
&\leq {\abs{\phi_n\pare{\frac{j+1}{n}}-\phi_n\pare{\frac{j}{n}}}}+\abs{\avint_{B^n_j}\phi(z)dz-\phi(x)}\\[10pt]
&\displaystyle =\abs{\kint_{B_n^{j+1}}{\phi(z)}-\kint_{B_n^{j}}{\phi(z)}}dz+\abs{\kint_{B_n^j}\phi(z)dz-\phi(x)}\\[10pt]
&\displaystyle \leq \kint_{B_n^{j}}\abs{\phi(z)-\phi\pare{z+\frac{1}{n}}}dz+\kint_{B_n^j}\abs{\phi(z)-\phi(x)}dz\\[10pt]
&\leq \epsilon,
\end{align}
for all $n>\delta_{\frac{\epsilon}{2}}$.
By \eqref{pl1} and \eqref{pl2} we get that there exists $n_0\pare{\epsilon}$ such that $\forall n\geq n_0\pare{\epsilon}$ 
\begin{align}\label{moni2}
\norm{\phi_n-\phi}_\infty<\epsilon.
\end{align}
}
\end{remark}

{\bf The chessboard case.} 

Now we consider in $\mathbb{R}^2$ the set $\overline\Omega = [0,1]\times [0,1]$ and divide (as we did in 1-d) $[0,1]$ into
$n$ subintervals that we call $\{I_n^j\}_{j\in \brc{1, \ldots, n}}$, each one of the same 
length, $\frac{1}{n}$. Now, we take and $A_n$ the collection of small squares that has $x\in I_n^j$ and $y \in I_n^i$ 
with both $j,i$ even or odd simultaneously ; while
$B_n$ is the complement (that is, $x\in I_n^j$ and $y\in I_n^i$ are inside odd and even subintervals).
In this case we also obtain $\theta =1/2$ and the limit system is given by \eqref{sys1.k.}.

{\bf Small balls into small squares.}

The previous example can be modified, considering as $B_n$ the union of small balls of radius $r/n$, with $r<1/2$
centered inside the 
small squares of side of length $1/n$
(that the radius is smaller than a half of the length of the side of the small quire is needed to have disjoint 
small balls that are the connected components of $B_n$). In this case we obtain that $\theta$ is also constant and is given by the proportion of the small square that is occupied
by the ball
$$
\theta = k = \frac{|B_r(0)|}{|Q_1|} = \pi r^2 <\frac{\pi}{4} <1 .
$$

\subsection{Thin strips}

Now, we want to present an example in which a second derivative in one direction survives
from the Laplacian part of the operator. This example shows that our condition on the size of the diameters
of the connected components of $B_n$ is necessary to obtain Theorem \ref{teo.main.intro}. 

We can consider the following configuration. In $[0,1]\times [0,1] \subset \mathbb{R}^2$ take thin vertical strips as the sets
$A_n$ and $B_n$, that is,
$$
A_n = \bigcup_{j=1, \, j \mbox{ even}}^n [1/(j-1), 1/j) \times [0,1]
$$
and
$$
B_n = \bigcup_{j=1, \, j \mbox{ odd}}^n [1/(j-1), 1/j) \times [0,1].
$$
This is a partition of $\overline\Omega$ into disjoint subsets $A_n$, $B_n$ that are narrow strips of width $1/n$.
Notice that
$$
\chi_{A_n} (x,y) \rightharpoonup \frac12 \qquad \mbox{and} \qquad  \chi_{B_n} (x,y) \rightharpoonup \frac12.
$$
Hence we also have
$$
\theta \equiv \frac12
$$
in this case.

Also remark that in this example the condition 
$$
  \max_j \left\{\mbox{diam} (B^j_n) \right\} \to 0, \mbox{ as } n \to \infty
$$
does not hold.

Here, the generator associated with our process $X_n$ is given by
\begin{align}\label{gen.example.strips}
{L}_nf(x,y)=&\chi_{A_n}\pare{x,y}\int_\Omega J\pare{(x,y),(z,w)}\pare{f\pare{z,w}-f\pare{x,y}}\, dz \, dw\\[10pt]
&+\chi_{B_n}\pare{x,y}\int_\Omega \chi_{A_n}\pare{z,w}J\pare{(x,y),(z,w)}\pare{f\pare{z,w}-f\pare{x,y}}\, dz \, dw+\frac{1}{2}\chi_{B_n}\pare{x,y}\Delta f \pare{x,y}.
\end{align}

In this case, we can also pass to the limit (as we did before) and obtain
that, as $n\to \infty$,
\begin{equation}\label{limite.debil.u.intro.99}
\begin{array}{l}
\displaystyle u_n(x,y,t) \rightharpoonup u(x,y,t),\qquad \mbox{weakly in } L^2 (\Omega \times (0,T)), \\[10pt]
\displaystyle \chi_{B_n} (x,y) u_n(x,y,t) \rightharpoonup a(x,t),\qquad \mbox{weakly in } L^2 (\Omega \times (0,T)), \\[10pt]
\displaystyle \chi_{A_n} (x,y) u_n(x,y,t) \rightharpoonup b(x,t),\qquad \mbox{weakly in } L^2 (\Omega \times (0,T)). 
\end{array}
\end{equation}
These limits verify
$$
u(x,y,t) = a(x,y,t) + b(x,y,t)
$$
and are characterized by
the fact that $(a,b)$ is the unique solution to the following system,
\begin{equation}\label{sys1.intro.44}
\left\{
\begin{array}{l}
\displaystyle \frac{\partial a}{\partial t}\pare{x,y, t}=\int_{\Omega}J(\pare{x, y},(z,w))\pare{a\pare{y,w,t}-a\pare{x,y,t}}dzdw
\\[10pt]
\displaystyle \qquad 
-\frac12 \int_{\Omega}J(\pare{x, y},(z,w)) a\pare{z,w,t}dzdw+ \frac12
\int_{\Omega}J(\pare{x,y},(z,w)) b\pare{z,w,t}dzdw  \\[10pt]
\displaystyle \frac{\partial b}{\partial t}\pare{x,y, t}=\frac12 \int_{\Omega} J(\pare{x, y},(z,w)) a\pare{z,w,t}dzdw
-b\pare{x,y,t} \frac12 \int_{\Omega} 
J(\pare{x, y},(z,w)) dzdw + \frac14 \frac{\partial^2 b}{\partial y^2} (x,y,t)
\\[10pt]
\displaystyle   - \frac{\partial b}{\partial y} (x,0,t) = \frac{\partial b}{\partial y} (x,1,t) = 0,
\\[10pt]
\displaystyle a\pare{x,y, 0}=\frac12 u_0\pare{x,y}, \quad b\pare{x,y, 0}=\frac12 u_0\pare{x,y},
\end{array} \right.
\end{equation}
with $(x,y)\in \Omega,$ $t>0$.

Notice that now a second derivative in the $y$ direction survives in the second equation. 

The only modification needed in the proof is to consider, as approximating sequence $\phi_n(x,y)$
$$
\phi_n\pare{x,y, t}=
\left\{
\begin{array}{ll} \displaystyle 
\kint_{B_n^j}\phi(z,y, t)dz &\text{if }(x,y)\in B_n^j, t\geq 0\\[10pt]
\phi\pare{x,y, t} &\text{if }(x,y) \in A_n^j, t\geq 0.
\end{array}
\right.
$$
Notice that for these functions $\phi_n$ we have
$$
\chi_{B_n} (x,y) \frac{\partial^2 \phi_n}{\partial x^2} \pare{x,y,t} \equiv 0, \qquad \mbox{ but } \qquad
\chi_{B_n} (x,y) \frac{\partial^2 \phi_n}{\partial y^2} \pare{x,y,t} \equiv \chi_{B_n} (x,y) \frac{\partial^2 \phi}{\partial y^2} \pare{x,y,t},
$$
and hence
$$
\frac12 \chi_{B_n} (x,y)  \Delta \phi_n (x,y,t) = \frac12 \chi_{B_n} (x,y) \frac{\partial^2 \phi}{\partial y^2} \pare{x,y,t} \rightharpoonup 
\frac14 \frac{\partial^2 \phi}{\partial y^2} \pare{x,y,t}.
$$
Therefore, the second derivative in $y$ survives in the equation for $b$.

All the other terms that appear in the equations can be handled as we did before. For example, we have
$$
\begin{array}{l}
\displaystyle 
\int_0^T\int_\Omega \int_\Omega J(\pare{x, y}, (z,w)
\chi_{B_n}\pare{x,y}a_n\pare{z,w, t}\phi_n\pare{x,y, t}\, dydxdzdw dt 
\\[10pt]
\qquad \displaystyle \xrightarrow[n\to +\infty]{}\int_0^T\int_\Omega 
\int_\Omega J(\pare{x, y},(z,w)) \frac12 a\pare{z,w, t}\phi\pare{x,y, t}\, dydxdzdwdt.
\end{array}
$$

Finally, let us point out that when the thin strips are oriented according to the direction of a vector $v$ then it is
the second derivative in the direction of $v$ the one that survives. 

Concerning the limit process we can look at the limit problem \eqref{sys1.intro.44} as a system describing
the movement of a particle that, as before, has a label (white or black), the probability density of being white is given by $a(x,t)$ and 
the probability density of being black is $b(x,t)$. The transition probabilities between labels and the jumps inside the domain
are as before (for instance, a particle that is black at $x$ at time $t$ becomes white and jumps 
to $y$ with a probability $\int_{\Omega} 
J\pare{x, y}\pare{1-\theta (y)}dy$) but this time particles with a black label also move according to a one dimensional
Brownian motion in the direction of $v$ between two consecutive jumps (this is reflected in the fact that
a second derivative in the $v$ direction survives).

\medskip

%%%%%%%%%%%%%%%%%%%%%%%%%%%%%%%%%%%%%%%%%%%%%%%%%%%%%%%%%%%%%%%
\noindent{\large \textbf{Acknowledgments}}

\noindent We want to thank Ines Armendariz for several interesting discussions.

\noindent Partially supported by 

- CONICET grant PIP GI No 11220150100036CO
(Argentina), 

- UBACyT grant 20020160100155BA (Argentina), 

- Project MTM2015-70227-P (Spain).

%%%%%%%%%%%%%%%%%%%%

%\bibliographystyle{amsalpha}
%\bibliography{biblio}

\noindent\textbf{Address:}

{Monia Capanna and Julio D. Rossi
\hfill\break\indent
CONICET and Departamento  de Matem{\'a}tica, FCEyN,
Universidad de Buenos Aires, Ciudad Universitaria, Pabellon I, (1428).
Buenos Aires, Argentina.}
\hfill\break\indent
{{\tt moniacapanna@gmail.com, jrossi@dm.uba.ar}}

\end{document}